\documentclass[letterpaper,journal,twoside,onecolumn,11pt]{IEEEtran}

\usepackage[utf8]{inputenc}
\usepackage[T1]{fontenc}
\usepackage{graphicx}
\DeclareGraphicsExtensions{.pdf}
\usepackage{xcolor}
\usepackage[cmex10]{amsmath}
\usepackage{amsthm,amssymb,amsfonts,mathrsfs,bm,mathtools,cases}
\makeatletter
\let\MYcaption\@makecaption
\makeatother
\usepackage[font=footnotesize]{subcaption}
\makeatletter
\let\@makecaption\MYcaption
\makeatother
\usepackage[sort,compress]{cite}
\usepackage{etoolbox}
\usepackage[linesnumbered,vlined,ruled,commentsnumbered]{algorithm2e}
\usepackage{xspace}
\usepackage{combelow}

\usepackage{setspace}
\makeatletter

\newcommand*{\ie}{%
  \@ifnextchar{,}%
  {\textit{i.e.}}%
  {\textit{i.e.,}\@\xspace}%
}
\newcommand*{\eg}{%
  \@ifnextchar{,}%
  {\textit{e.g.}}%
  {\textit{e.g.,}\@\xspace}%
}
\newcommand*{\etc}{%
  \@ifnextchar{.}%
  {\textit{etc}}%
  {\textit{etc.}\@\xspace}%
}
\newcommand*{\etal}{%
  \@ifnextchar{.}%
  {\textit{et al}}%
  {\textit{et al.}\@\xspace}%
}
\newcommand*{\cf}{%
  \@ifnextchar{.}%
  {\textit{cf}}%
  {\textit{cf.}\@\xspace}%
}
\newcommand*{\aka}{%
  \@ifnextchar{,}%
  {\textit{a.k.a.}}%
  {\textit{a.k.a.}\@\xspace}%
}
\makeatother

\theoremstyle{plain}
\newtheorem{thm}{Theorem}
\newtheorem{lemma}{Lemma}
\newtheorem{assumption}{Assumption}

\newtheorem{fact}{Fact}

\newtheorem{coroll}{Corollary}

  {\end{enumerate}}

\newenvironment{lemlist}{%

  \begin{enumerate}}%
  {\end{enumerate}}

\newenvironment{asslist}{%

  \begin{enumerate}}%
  {\end{enumerate}}

  {\end{enumerate}}

  {\end{enumerate}}

  {\end{enumerate}}

  {\end{enumerate}}

  {\end{enumerate}}

  {\end{enumerate}}

  {\end{enumerate}}

  {\end{enumerate}}

\newenvironment{applist}{%

  \begin{enumerate}}%
  {\end{enumerate}}

\usepackage[sort]{cleveref}

\crefname{thm}{Theorem}{Theorems}
\crefname{prop}{Proposition}{Propositions}
\crefname{assumption}{Assumption}{Assumptions}
\crefname{lemma}{Lemma}{Lemmata}
\crefname{definition}{Definition}{Definitions}
\crefname{example}{Example}{Examples}
\crefname{algo}{Algorithm}{Algorithms}
\crefname{fact}{Fact}{Facts}
\crefname{claim}{Claim}{Claims}
\crefname{appendix}{Appendix}{Appendices}
\crefname{coroll}{Corollary}{Corollaries}
\crefname{figure}{Figure}{Figures}
\crefname{section}{Section}{Sections}
\crefname{algorithm}{Algorithm}{Algorithms}

\makeatletter
\DeclareRobustCommand{\crefnosort}[1]{%
  \begingroup\@cref@sortfalse\cref{#1}\endgroup
}
\makeatother

\newcommand{\IntegerP}{\mathbb{Z}_{\geq 0}}
\newcommand{\IntegerPP}{\mathbb{Z}_{>0}}
\newcommand{\Real}{\mathbb{R}}

\newcommand{\RealPP}{\mathbb{R}_{>0}}

\newcommand\given{{\mathbin{}\mid\mathbin{}}}
\newcommand\vect[1]{\mathbf{#1}}
\newcommand\mypound{\protect\scalebox{0.8}{\protect\raisebox{0.4ex}{\#}}}
\newcommand\limas{\xrightarrow{\text{a.s.}}}

\providecommand\given{} 
\newcommand\SetSymbol[1][]{
  \nonscript\,#1\vert \allowbreak \nonscript\,\mathopen{}}
\DeclarePairedDelimiterX\Set[1]{\lbrace}{\rbrace}%
{\renewcommand\given{\SetSymbol[\delimsize]} #1 }
\DeclarePairedDelimiterX\innerp[2]{\langle}{\rangle}{#1
  \mathop{}\delimsize\vert\mathop{} #2}
\DeclarePairedDelimiterX\norm[1]\lVert\rVert{\ifblank{#1}{\:\cdot\:}{#1}}

\DeclareMathOperator{\sign}{sgn}

\DeclareMathOperator{\Fix}{Fix}

\DeclareMathOperator{\graph}{gph}

\DeclareMathOperator{\Expect}{\mathbb{E}}
\DeclareMathOperator{\Prob}{\mathbb{P}}

\DeclareMathOperator{\rank}{rank}

\DeclareMathOperator{\prox}{Prox}

\DeclareMathOperator{\Id}{Id}

\DeclareMathOperator{\range}{ran}

\begin{document}

\title{The Stochastic Fej\'{e}r-Monotone Hybrid Steepest Descent Method and the Hierarchical RLS}

\author{Konstantinos~Slavakis
  \thanks{K.~Slavakis is with the
    Department of Electrical Engineering, University at Buffalo, The State University of New York,
    NY 14260-2500, USA; Email: kslavaki@buffalo.edu; Tel: +1~(716)~645-1012.}%
  \thanks{Parts of this work were presented at \cite{SFMHSDM.Asilomar.18}. This work was supported
    by the NSF grants 1525194 and 1718796.}}

\maketitle

\begin{abstract}
  This paper introduces the stochastic Fej\'{e}r-monotone hybrid steepest descent method (S-FM-HSDM)
  to solve affinely constrained and composite convex minimization tasks. The minimization task is
  not known exactly; noise contaminates the information about the composite loss function and the
  affine constraints. S-FM-HSDM generates sequences of random variables that, under certain
  conditions and with respect to a probability space, converge point-wise to solutions of the
  noiseless minimization task. S-FM-HSDM enjoys desirable attributes of optimization techniques such
  as splitting of variables and constant step size (learning rate). Furthermore, it provides a novel
  way of exploiting the information about the affine constraints via fixed-point sets of appropriate
  nonexpansive mappings. Among the offsprings of S-FM-HSDM, the hierarchical recursive least squares
  (HRLS) takes advantage of S-FM-HSDM's versatility toward affine constraints and offers a novel
  twist to LS by generating sequences of estimates that converge to solutions of a hierarchical
  optimization task: Minimize a convex loss over the set of minimizers of the ensemble LS
  loss. Numerical tests on a sparsity-aware LS task show that HRLS compares favorably to several
  state-of-the-art convex, as well as non-convex, stochastic-approximation and online-learning
  counterparts.
\end{abstract}


\section{Introduction}\label{Sec:intro}

\subsection{Problem statement}\label{Sec:problem}

\IEEEPARstart{T}{he} following problem is considered: With a stochastic oracle providing estimates
$f_n$ (or even $\nabla f_n$), $h_n$ and $\mathcal{A}_n$ per $n$ ($n$ denotes discrete time
\textit{and}\/ iteration index; $n\in \IntegerPP \coloneqq \Set{1,2, \ldots}$) of the generally
unknown convex functions $f, h$ and the affine set $\mathcal{A}$, respectively, solve
\begin{align}
  \min\nolimits_{x\in\mathcal{A} \subset\mathcal{X}} f(x) + h(x) + g(x)\,,
  \tag{P} \label{the.problem}
\end{align}
where $\mathcal{X}$ is a finite-dimensional real Hilbert space. Only the convex (regularizing)
function $g$ is assumed to be known exactly. The goal is to construct a sequence of estimates
$(x_n)_n \coloneqq (x_n)_{n\in\IntegerP} \subset\mathcal{X}$ by exploiting the information about
$(f_n)_n$, or $(\nabla f_n)_n$, $(h_n, \mathcal{A}_n)_n$ as well as $g$, and to identify the
conditions which ensure, despite the uncertainty about $f, h$ and $\mathcal{A}$, the point-wise
convergence of $(x_n)_n$ to a solution of \eqref{the.problem} with respect to (w.r.t.) a probability
space.

Instances of \eqref{the.problem} appear in adaptive filtering (AF)~\cite{SayedBook,
  Kalou.Theodo.book.93, Pereyra.survey.16}; in particular, in linear equalization, channel
estimation, beamforming, tracking of fading channels, line and acoustic echo cancellation and active
noise control~\cite{SayedBook}. Special cases of \eqref{the.problem} appear also in stochastic
approximation (SA)~\cite{Kushner.Yin, Pereyra.survey.16} and online learning (OL)~\cite{OCO.book.12,
  Pereyra.survey.16} as in training artificial neural networks, learning optimal strategies in
Markov decision processes, recursive games, sequential-decision tasks in
economics~\cite{Kushner.Yin}, online classification and multi-armed bandit
problems~\cite{OCO.book.12}. (An outline of the strong ties and distinct differences between SA and
online learning is provided in \cite{VisAVis.14}.)

Each one of the three loss terms in \eqref{the.problem} plays a distinct role: $f$ is smooth and
generally unknown, $h$ can be non-smooth and unknown, while $g$ comprises all \textit{known}\/ and
possibly non-smooth regularizing losses. The affine constraint $\mathcal{A}$ renders
\eqref{the.problem} a versatile framework that encompasses a large variety of problems. For example,
given the finite-dimensional Hilbert spaces $\Set{\mathsf{X}_k}_{k=0}^{I_h+J_g}$, with
$I_h, J_g\in \IntegerPP$, the convex functions $\mathfrak{f}: \mathsf{X}_0\to\Real$,
$\mathfrak{h}^{(i)}: \mathsf{X}_i\to\Real \cup \Set{+\infty}$,
$\mathfrak{g}^{(j)}: \mathsf{X}_{j+I_h}\to\Real \cup \Set{+\infty}$, the linear mappings
$H^{(i)}: \mathsf{X}_0\to \mathsf{X}_i$, $G^{(j)}: \mathsf{X}_0\to \mathsf{X}_{j+I_h}$,
$\mathsf{p}^{(i)}\in \mathsf{X}_i$ and $\mathsf{q}^{(j)}\in \mathsf{X}_{j+I_h}$, for
$i\in\Set{1, \ldots, I_h}$ and $j\in\Set{1, \ldots, J_g}$, then the highly structured composite
problem
\begin{align}
  \min_{\mathsf{x}^{(0)} \in\mathsf{X}_0} \mathfrak{f}(\mathsf{x}^{(0)})
  & + \sum\nolimits_{i=1}^{I_h}
    \mathfrak{h}^{(i)}(H^{(i)}\mathsf{x}^{(0)} - \mathsf{p}^{(i)}) \notag\\
  & + \sum\nolimits_{j=1}^{J_g} \mathfrak{g}^{(j)}(G^{(j)}\mathsf{x}^{(0)} -
    \mathsf{q}^{(j)}) \label{problem.multiple.h.g}
\end{align}
can be recast as \eqref{the.problem} if
$\mathcal{X} \coloneqq \mathsf{X}_0 \times \mathsf{X}_1 \times \cdots \times \mathsf{X}_{I_h+J_g} =
\Set{x\coloneqq (\mathsf{x}^{(0)}, \ldots, \mathsf{x}^{(I_h+J_g)}) \given \mathsf{x}^{(k)}\in
  \mathsf{X}_k, \forall k\in \Set{0, \ldots, I_h+J_g}}$,
$f(x)\coloneqq \mathfrak{f}(\mathsf{x}^{(0)})$,
$h(x)\coloneqq \sum_{i=1}^{I_h} \mathfrak{h}^{(i)}(\mathsf{x}^{(i)})$,
$g(x)\coloneqq \sum_{j=1}^{J_g} \mathfrak{g}^{(j)}(\mathsf{x}^{(j+I_h)})$, and the closed affine set
$\mathcal{A}\coloneqq \Set{x\in \mathcal{X} \given \mathsf{x}^{(0)}\in \mathsf{X}_0,
  \mathsf{x}^{(i)} = H^{(i)} \mathsf{x}^{(0)}- \mathsf{p}^{(i)}, \mathsf{x}^{(j+I_h)} = G^{(j)}
  \mathsf{x}^{(0)} - \mathsf{q}^{(j)}, i\in\Set{1, \ldots, I_h}, j\in\Set{1, \ldots, J_g}}$. The
splitting of variables via Cartesian-product spaces facilitates processing; \eg
\eqref{split.variables.LS.problem}. The \eqref{the.problem} formulation can also accommodate any
closed convex (not necessarily affine) constraint $\mathcal{C}$ as follows: Consider
\eqref{problem.multiple.h.g} and let one of the $\Set{\mathfrak{h}^{(i)}}_i$ or
$\Set{\mathfrak{g}^{(j)}}_j$, depending on whether $\mathcal{C}$ bears stochasticity or not, take
the form of the indicator function $\iota_{\mathcal{C}}$ (see \cref{app:preliminaries} for the
definition). More importantly, \eqref{the.problem} allows for cases where the information about
$\mathcal{A}$ is not known exactly, introduces thus stochasticity into $\mathcal{A}$ and opens the
door to new problem formulations and novel algorithmic developments, \eg \eqref{HLS} and
\cref{algo:HRLSa}.

\subsection{Case study: Sparsity-aware least squares}\label{Sec:system.id}

To highlight the versatility of \eqref{the.problem} and to unfold all features of the proposed
algorithmic solution, coined stochastic Fej\'{e}r-monotone hybrid steepest descent method
(S-FM-HSDM), it is instructive to build the discussion around specific instances of
\eqref{the.problem}. To this end, let $\mathcal{X}$ be the Euclidean $\Real^D$. Bold-faced symbols
indicate that $\mathcal{X} = \Real^D$; in particular, lowercase bold-faced symbols denote vectors in
$\Real^D$. Consider a sparse system $\bm{\theta}_*\in \mathcal{X}$ and the classical
\textit{linear-regression model:} $b_n = \vect{a}_n^{\intercal} \bm{\theta}_* + \eta_n$, almost
surely (a.s.), $\forall n\in\IntegerPP$, with input-output data pair
$(\vect{a}_n, b_n)\in \mathcal{X}\times \Real$, the noise process $(\eta_n)_n$ is assumed to be
zero-mean and independent of $(\vect{a}_n)_n$, and $\intercal$ denotes vector/matrix
transposition. Typical stationarity assumptions on $(\vect{a}_n, b_n)_n$ are adopted also here:
$\vect{R} \coloneqq \Expect (\vect{a}_n \vect{a}_n^{\intercal})$,
$\vect{r} \coloneqq \Expect(b_n \vect{a}_n)$, and $\Expect(b_n^2)$ stay constant $\forall n$, where
$\Expect(\cdot)$ denotes expectation. It is well-known that $\bm{\theta}_*$ satisfies the normal
equations
$\bm{\theta}_*\in \Set{\vect{x}\in \mathcal{X} \given \vect{Rx} =
  \vect{r}}$~\cite[(3.9)]{SayedBook}. This section deals with the system-identification problem of
estimating the sparse $\bm{\theta}_*$ without knowing $(\vect{R}, \vect{r})$ but relying only on the
information $(\vect{a}_n, b_n)_n$ provided by the stochastic oracle.

Motivated by the celebrated (Lagrangian form of the) least absolute shrinkage and selection operator
(LASSO)~\cite[(3.52)]{Hastie.book}, designed to solve sparse system-identification problems, the
first instance of \eqref{the.problem} is the convexly regularized least squares:
$\forall n\in\IntegerPP$,
\begin{align}
  & \min_{\vect{x}\in\Real^D} \overbrace{\tfrac{1}{2} \vect{x}^{\intercal} \vect{R}
    \vect{x} - \vect{r}^{\intercal} \vect{x} + \tfrac{1}{2}\Expect(b_n^2)}^{l(\vect{x})} +
    \overbrace{\rho\norm{\vect{x}}_1}^{g(\vect{x})} \notag\\
  & = \min_{\vect{x}\in\Real^D} \Expect \Bigl[\;\underbrace{\tfrac{1}{2n}
    \sum\nolimits_{\nu=1}^n \left(\vect{a}_{\nu}^{\intercal} \vect{x} -
    b_{\nu}\right)^2}_{l_n(\vect{x})}\;\Bigr] + \rho\norm{\vect{x}}_1 \,,
    \tag{CRegLS} \label{CRegLS}
\end{align}
where the $\ell_1$-norm regularizer promotes sparse solutions. \eqref{CRegLS} becomes a special case
of \eqref{the.problem}, if $\mathcal{A} \coloneqq \mathcal{X} = \Real^D \eqqcolon \mathcal{A}_n$,
$(f, f_n) \coloneqq (l, l_n)$, or, $(h, h_n) \coloneqq (l, l_n)$ a.s.

The second instance of \eqref{the.problem} exploits the fact that even the information about
$\mathcal{A}$ may be inexact, and takes the form of a \textit{hierarchical}\/ (H)LS estimation task,
which appears to be new in the AF, SA and OL literature: $\forall n$,
\begin{align}
  \min_{\vect{x}\in\Real^D}
  & {} \left[\, \norm{\vect{x}}_1 \eqqcolon g(\vect{x}) \,\right]\notag \\
  \text{s.to}\
  & {} \vect{x}\in \underbrace{\arg\min_{\vect{x}' \in\Real^D} \Expect \Bigl[
    \sum\nolimits_{\nu=1}^n \left(\vect{a}_{\nu}^{\intercal} \vect{x}' - b_{\nu}\right)^2
    \Bigr]}_{\mathcal{A}} \,, \tag{HLS} \label{HLS}
\end{align}
\ie, the convex loss $g(\cdot)$, here $\norm{}_1$, is minimized over the set of minimizers of the
classical (ensemble) LS loss. Recall that $\mathcal{A}$ in \eqref{HLS} comprises all vectors,
including $\bm{\theta}_*$, that satisfy the normal equations. In the case of
$g(\cdot) \coloneqq \norm{}_1$, \eqref{HLS} can be also viewed as an SA extension of (the
\textit{deterministic}) basis pursuit~\cite{Chen.bp.01}. The mainstream approach, \eg,
\cite{Pereyra.survey.16, PLC.JCP.stochastic.QF.15, Rosasco.Optim.16}, to deal with \eqref{HLS} is to
employ the indicator function $\iota_{\mathcal{A}}$ in the place of one of the $\mathfrak{h}^{(i)}$
and $\mathfrak{g}^{(j)}$ in \eqref{problem.multiple.h.g}. Such a path restricts the means of
treating $\mathcal{A}$ to the projection mapping $P_{\mathcal{A}}$ [recall that $P_{\mathcal{A}}$ is
the proximal mapping of $\iota_{\mathcal{A}}$; \cf~\eqref{def.proximal}]. Since
$(\vect{R}, \vect{r})$ are generally unknown, $\mathcal{A}$ is also unknown to the user. Still, the
goal is to solve \eqref{HLS}. If $(f, f_n) \coloneqq (0, 0) \eqqcolon (h, h_n)$, and $\mathcal{A}_n$
is defined as an estimate of $\mathcal{A}$, then \eqref{HLS} turns out to be a special instance of
\eqref{the.problem}. This paper provides a novel way of using the available estimates
$(\mathcal{A}_n)_n$ of $\mathcal{A}$ via fixed-point sets of appropriate nonexpansive mappings (\cf
\cref{Sec:algo}). This new viewpoint pays off in the computationally efficient HRLSa (\cf
\cref{algo:HRLSa}), which solves \eqref{HLS} under certain conditions, despite the uncertainty in
the estimates $(\mathcal{A}_n)_n$, while scoring the lowest estimation error across a variety of
numerical-test scenarios versus several state-of-the-art schemes (\cf \cref{Sec:tests}).

\subsection{Prior art}\label{Sec:prior.art}

In most cases, OL and SA algorithms have their origins in deterministic optimization schemes. For
example, the OL scheme \cite{ProxSVRG} draws inspiration from the forward-backward (a.k.a.\
proximal-gradient) algorithm~\cite[\S27.3]{HB.PLC.book} and incorporates variance-reduction
arguments~\cite{SVRG.13} into its iterations to effect convergence speed-ups in solving a special
case of \eqref{the.problem}, which appears to be of primary importance in machine learning:
$f \coloneqq (1/M) \sum_{m=1}^M \mathfrak{f}^{(m)}$, where $\Set{\mathfrak{f}^{(m)}}_{m=1}^M$ are
convex and smooth, $M\in\IntegerPP$ is very large, $h \coloneqq 0$ and
$\mathcal{A} \coloneqq \mathcal{X}$. Driven by the need to avoid the cumbersome computation of
$\nabla f$, stochasticity is introduced by selecting randomly only a small subset of
$\Set{\mathfrak{f}^{(m)}}_{m=1}^M$, per time/iteration index $n$, to form an estimate of $\nabla
f$. Recent SA schemes, motivated by the forward-backward algorithm and formulated in the more
general setting of monotone-operator inclusions, can be found in \cite{PLC.JCP.stochastic.QF.15,
  PLC.JCP.stochastic.FB.16, Bianchi.FB.17}. An SA extension of primal-dual methods, where
stochasticity is introduced via general sampling techniques to deal with massive data, is reported
in~\cite{Chambolle.stochastic.PD.18}. An SA extension of the Douglas-Rachford
algorithm~\cite[\S25.2, \S27.2]{HB.PLC.book} is reported in~\cite{PLC.JCP.stochastic.QF.15}. Study
\cite{FastLightSADMM} extends the celebrated alternating direction method of multipliers (ADMM) to
the OL setting, and blends it with variance-reduction arguments to solve a problem similar to that
of \cite{ProxSVRG}, but with a non-trivial, yet deterministic affine constraint
$\mathcal{A} \subsetneq \mathcal{X}$. Furthermore, \cite{Flammarion.Bach.17a} explores the
dual-averaging scheme of \cite{Nesterov.PD.09} in the SA context offering linear-convergence
guarantees for a quadratic $f$ in \eqref{the.problem}, while $\mathcal{X}$ is a closed convex set
with non-empty interior. Moreover, the SA schemes~\cite{Lan.ACSA.12, Rosasco.Optim.16} are motivated
by the deterministic acceleration method of~\cite{Nesterov.83}; in particular, \cite{Lan.ACSA.12}
uses specific step sizes (\cf \cite[(33)]{Lan.ACSA.12}) to effect convergence acceleration in the
case where $h\coloneqq0$, $g$ is (Lipschitz) continuous and a deterministic convex compact
constraint takes the place of $\mathcal{A}$ in \eqref{the.problem}.

With regards to the specific setting of \cref{Sec:system.id}, the state-of-the-art AF schemes
\cite{RLS.meets.l1, SPARLS.10, l0RLS} are built around a variation of \eqref{CRegLS}, where the
regularizing coefficient $\rho_n$ converges to zero as $n\to\infty$. A Bayesian approach to the LS
sparse system-identification problem appears in \cite{Themelis.ASVB.14}, and a greedy RLS approach
based on the orthogonal-matching-pursuit algorithm is reported in \cite{Dumitrescu.greedy.RLS.12}. A
majorization-minimization approach, which includes also non-convex regularizers, is studied
in~\cite{Chouzenoux.stochastic.MM.17}. Basis pursuit~\cite{Chen.bp.01} is used in
\cite{Benesty.BP.10} to provide an interpretation of the estimate-update equation per iteration $n$
of several proportionate-type AF schemes; however, an ensemble-based viewpoint, such as \eqref{HLS},
and a performance analysis are not provided.

\subsection{Contributions and structure of the manuscript}\label{Sec:contributions}

Similarly to \cite{ProxSVRG, PLC.JCP.stochastic.QF.15, PLC.JCP.stochastic.FB.16, Bianchi.FB.17,
  FastLightSADMM, Flammarion.Bach.17a, Lan.ACSA.12}, the proposed S-FM-HSDM (\cref{algo:SFMHSDM})
springs from the deterministic FM-HSDM \cite{FM-HSDM.Optim.18}, which belongs to the HSDM
family~\cite{Yamada.HSDM.2001} and solves \eqref{the.problem} in infinite-dimensional Hilbert spaces
with no stochasticity involved. In \cite{FM-HSDM.Optim.18}, the information about the affine
constraint $\mathcal{A}$ is incorporated into FM-HSDM via an affine nonexpansive mapping
$T: \mathcal{X}\to \mathcal{X}$ whose fixed-point set is
$\mathcal{A} = \Fix T \coloneqq \Set{x\in\mathcal{X} \given Tx = x}$. For example, the (metric)
projection mapping $P_{\mathcal{A}}$ onto $\mathcal{A}$ (\cf \cref{app:preliminaries}) may serve as
$T$~\cite[Prop.~2.11]{FM-HSDM.Optim.18}. Interestingly, the versatile \cite{FM-HSDM.Optim.18} allows
for numerous choices of $T$ other than the mainstream $P_{\mathcal{A}}$ [\cf~\eqref{T.family}].

S-FM-HSDM extends FM-HSDM to the stochastic setting.  With a stochastic oracle providing a sequence
of affine constraints $(\mathcal{A}_n)_n$ as estimates of the generally unknown $\mathcal{A}$, a
mapping $T_n$ is chosen per time index $n$, with $\mathcal{A}_n = \Fix T_n$, to serve as an estimate
of $T$. There are numerous choices of $T_n$ other than the obvious $P_{\mathcal{A}_n}$. Furthermore,
$f$ and $h$ are not required to be known exactly and only estimates $(f_n)_n$ [or even
$(\nabla f_n)_n$] and $(h_n)_n$ are provided to the user by the stochastic oracle. The versatility
of S-FM-HSDM is demonstrated in the system-identification context of \cref{Sec:system.id}, where
S-FM-HSDM solves \eqref{HLS} in \cref{Sec:algo}, with its specific form coined \textit{hierarchical
  (H)RLS.} Mappings $(T_n)_n$ drive the HRLS iterates asymptotically to a vector in $\mathcal{A}$,
and HRLS solves \eqref{HLS} \textit{without}\/ employing any sub-routines for identifying
$\mathcal{A}$ prior to minimizing $g$ over $\mathcal{A}$. It is worth recalling here that
identifying $\mathcal{A}$ requires the computation of $\Expect(\cdot)$ which is a usually
intractable task for the user. A specific choice of $T_n$ [\cf \eqref{Tn.grad}] yields the
computationally efficient HRLSa flavor of S-FM-HSDM (\cf \cref{Sec:algo}).

Many SA methods, such as the classical~\cite{Robbins.Monro.51} and its convex-analytic
extension~\cite{Rosasco.SPGD.14}, rely on diminishing step sizes (learning rates) to ensure a.s.\
convergence of their iterates. Nevertheless, constant step-size schemes, \eg, \cite{Bianchi.FB.17,
  PLC.JCP.stochastic.QF.15}, are highly desirable in signal processing and machine learning since
they appear to \textbf{(i)} reach the neighborhood of solutions in a fewer number of iterations than
the diminishing step-size methods~\cite{Bianchi.FB.17}; and \textbf{(ii)} adapt quickly to changes
of non-stationary environments and track dynamically changing sets of solutions (\cf
\cref{fig:Sce11} and~\cite[Ch.~21]{SayedBook}). S-FM-HSDM operates with a constant step size
$\forall n$. The performance analysis of \cref{Sec:under.the.hood} identifies those conditions which
ensure that S-FM-HSDM converges a.s.\ to a solution of \eqref{the.problem}. For clarity, those
conditions are exemplified in the context of \cref{Sec:system.id}.

To validate the theoretical developments of this work, extensive numerical tests on synthetic data,
within the context of \cref{Sec:system.id}, are reported in \cref{Sec:tests}. Flavors HRLSa and
HRLSb of S-FM-HSDM appear to be the most consistent methods in achieving the lowest estimation error
across a variety of scenarios versus several state-of-the-art AF, SA and OL schemes.

To improve readability, S-FM-HSDM, its specific flavors within the context of \cref{Sec:system.id}
and their main theoretical results are presented first in \cref{Sec:algo}. The performance analysis
and the accompanying assumptions are detailed in \cref{Sec:under.the.hood}, while the necessary
mathematical preliminaries and proofs are deferred to the appendices.

\section{The S-FM-HSDM Family and Its Properties}\label{Sec:algo}

\subsection{The user-defined mappings $(T_n)_n$}\label{Sec:Nonexp.Tn}

To utilize the information about $\mathcal{A}$, this work follows \cite{FM-HSDM.Optim.18} and
considers a mapping $T$ s.t.\ $\Fix(T) = \mathcal{A}$. An obvious choice for $T$ would be the
(metric) projection mapping $P_{\mathcal{A}}$ onto
$\mathcal{A}$~\cite[Prop.~2.11]{FM-HSDM.Optim.18}. Nevertheless, this study revolves around less
obvious cases. In the context of \cref{Sec:system.id}, such examples are:
\begin{subnumcases}{\label{T.LS} T=}
  (\vect{I} - \tfrac{\mu}{\varpi}\vect{R}) + \tfrac{\mu}{\varpi} \vect{r}\,, \quad \varpi\geq
  \norm{\vect{R}}\,,\ \mu\in (0,1]\,, & \label{T.grad}\\
  (\vect{I} + \kappa\vect{R})^{-1} + \kappa (\vect{I} + \kappa\vect{R})^{-1} \vect{r} \,, \quad
  \kappa\in\RealPP\,, & \label{T.prox}
\end{subnumcases}
where $\vect{I}$ is the identity matrix and the spectral norm $\norm{\vect{R}}$ is equal to the
maximum eigenvalue of $\vect{R}$. In fact, any mapping which belongs to the following family of
mappings may serve as a candidate for $T$ \cite[Prop.~2.11]{FM-HSDM.Optim.18}:
\begin{align}\label{T.family}
  \mathfrak{T}_{\mathcal{A}} \coloneqq \Set*{T:\mathcal{X}\to
  \mathcal{X} 
  \given \begin{aligned}
    & \Fix T = \mathcal{A}; T = Q + \pi; \\
    & Q\ \text{is positive}; \norm{Q}\leq 1; \pi\in\mathcal{X}\\
  \end{aligned}}\,.
\end{align}
Any $T\in \mathfrak{T}_{\mathcal{A}}$ is affine, \ie, there exists a linear mapping
$Q:\mathcal{X}\to\mathcal{X}$ and a $\pi\in\mathcal{X}$ s.t.\ $Tx = Qx + \pi$,
$\forall x \in \mathcal{X}$; in short, $T = Q + \pi$. For the linear $Q$,
$\norm{Q} \coloneqq \sup_{\Set{x \given\,\norm{x} \leq 1}} \innerp{x}{Qx}$. Mapping
$Q: \mathcal{X}\to\mathcal{X}$ is called positive if it is linear, bounded, self-adjoint and
$\innerp{x}{Qx}\geq 0$, $\forall x\in\mathcal{X}$~\cite[\S9.3]{Kreyszig}. Since $\norm{Q}\leq 1$,
every mapping $T\in \mathfrak{T}_{\mathcal{A}}$ turns out to be nonexpansive~\cite{HB.PLC.book}:
$\forall (x, x')\in \mathcal{X}^2$,
$\norm{Tx - Tx'} = \norm{Qx - Qx'} = \norm{Q(x-x')} \leq \norm{Q} \,\norm{x-x'} \leq
\norm{x-x'}$. It is also worth noticing here that $\mathfrak{T}_{\mathcal{A}}$ is closed under any
convex combination and certain compositions of its
members~\cite[Prop.~2.10]{FM-HSDM.Optim.18}. Notice also that
$P_{\mathcal{A}}\in \mathfrak{T}_{\mathcal{A}}$~\cite[Prop.~2.11]{FM-HSDM.Optim.18}. Further
information on $\mathfrak{T}_{\mathcal{A}}$ is deferred to \cref{app:preliminaries}.

Notwithstanding, $\mathcal{A}$ is in general unknown to the user; hence, so is $T$ as well. With the
stochastic oracle providing estimates $(\mathcal{A}_n)_n$ of $\mathcal{A}$, the user needs to
construct mappings $(T_n)_n$ that serve as estimates of the unknown $T$. In the context of
\cref{Sec:system.id}, for example, instead of the unknown $\vect{R}$ and $\vect{r}$, their classical
running-average estimates~\cite{SayedBook}, $\forall n\in\IntegerPP$,
\begin{align}
  \vect{R}_n \coloneqq \tfrac{1}{n} \sum\nolimits_{\nu=1}^n \vect{a}_{\nu}
  \vect{a}_{\nu}^{\intercal} \,, \quad \vect{r}_n \coloneqq \tfrac{1}{n} \sum\nolimits_{\nu=1}^n
  b_{\nu} \vect{a}_{\nu}\,, \label{Rn.rn}
\end{align}
can be used to define
\begin{subnumcases}{\label{Tn.LS} T_n \coloneqq }
  (\vect{I} - \tfrac{\mu}{\varpi_n}\vect{R}_n) + \tfrac{\mu}{\varpi_n} \vect{r}_n \,, 
  & \hspace{-70pt} $\varpi_n \geq \norm{\vect{R}_n}$\,, \notag\\
  & \hspace{-70pt} $\mu\in (0,1]$\,, \label{Tn.grad}\\
  (\vect{I} + \kappa\vect{R}_n)^{-1} + \kappa (\vect{I} + \kappa\vect{R}_n)^{-1}
  \vect{r}_n \,,\ \kappa\in\RealPP\,. & \label{Tn.prox}
\end{subnumcases}

\begin{lemma}\label{lem:T.LS}
  For the affine set $\mathcal{A} \coloneqq \Set{\vect{x} \given \vect{Rx} = \vect{r}}$ in
  \cref{Sec:system.id}, mappings \eqref{T.LS} belong to $\mathfrak{T}_{\mathcal{A}}$. Moreover,
  mappings \eqref{Tn.LS} take the form $T_n = Q_n + \pi_n$, where $Q_n: \mathcal{X} \to \mathcal{X}$
  is positive, with $\norm{Q_n} \leq 1$, and $\pi_n\in \mathcal{X}$, a.s., $\forall n$.
\end{lemma}

\begin{proof}
  See \cref{app:T.LS}.
\end{proof}


\begin{algorithm}[t]
  \DontPrintSemicolon
  \SetKwInOut{input}{Stochastic oracle's input}
  \SetKwInOut{parameters}{User's input}
  \SetKwInOut{output}{Output}
  \SetKwBlock{initial}{Initialization}{}

  \input{$(\nabla f_n, h_n, \mathcal{A}_n)_{n\in\IntegerPP}$, $L_{\nabla f}$, $g$.}

  \parameters{$\alpha\in[0.5,1)$, $\lambda \in (0, 2(1-\alpha)/L_{\nabla f})$, $T_0$, $\nabla
    f_0$, $x_0$, and $(T_n)_{n\in\IntegerP}$.}

  \output{Sequence $(x_n)_{n\in\IntegerP}$.}

  \BlankLine
  \initial{%

    $ x_{1/2} \coloneqq T_0^{(\alpha)}x_0 - \lambda \nabla f_0(x_0)$. \label{SFMHSDM.step.half}
    
    $x_1 \coloneqq \prox_{\lambda (h_0+g)} (x_{1/2})$.\; \label{SFMHSDM.prox.step.1}
    
  }

  \For{$n = 1$ \KwTo $+\infty$}{%

    $x_{n+1/2} \coloneqq x_{n-1/2} - [T_{n-1}^{(\alpha)}x_{n-1} - \lambda \nabla f_{n-1}(x_{n-1})] +
    [ T_nx_n - \lambda \nabla f_n(x_n)]$. \label{SFMHSDM.step.n.half}
    
    $x_{n+1} \coloneqq \prox_{\lambda (h_n + g)} (x_{n+1/2})$. \label{SFMHSDM.prox.step.n}
    
  }

  \caption{S-FM-HSDM}\label{algo:SFMHSDM}
\end{algorithm}


\subsection{The S-FM-HSDM family}\label{Sec:SFMHSDM.family}

With mapping $T_n$ available, and with the \textit{averaged}\/ mapping $T_n^{(\alpha)}$ defined as
$T_n^{(\alpha)} \coloneqq \alpha T_n + (1-\alpha)\Id$, $\forall n$, where
$\Id: \mathcal{X}\to \mathcal{X}$ stands for the identity operator, S-FM-HSDM is presented in
\cref{algo:SFMHSDM}. $\prox$ in lines \ref{SFMHSDM.prox.step.1} and \ref{SFMHSDM.prox.step.n} of
\cref{algo:SFMHSDM} denotes the proximal mapping [\cf~\eqref{def.proximal}]. In the case where
$L_{\nabla f}$ is not available or cannot be estimated, S-FM-HSDM offers the option of setting
$(f, f_n) \coloneqq (0, 0)$, where $L_{\nabla f}$ can be set equal to any positive real-valued
number (\cf~\cref{Sec:tests}), and any estimate of $f$ can be transferred to the loss $h_n$,
since assumptions on $h$ and $h_n$ are weaker than those on $f$ and $f_n$
(\cf~\cref{Sec:under.the.hood}). Strategies for estimating $L_{\nabla f}$, in the case it is
unknown, will be reported elsewhere. Line~\ref{SFMHSDM.step.n.half} requires only the computation of
the current first-order information $\nabla f_n(x_n)$, whereas $\nabla f_{n-1}(x_{n-1})$, which was
computed at the previous time instance, can be pulled from a buffer that stores information.

\begin{algorithm}[t]
  \DontPrintSemicolon
  \SetKwInOut{input}{Stochastic oracle's input}
  \SetKwInOut{parameters}{User's input}
  \SetKwInOut{output}{Output}
  \SetKwBlock{initial}{Initialization}{}

  \input{$(\vect{a}_n, b_n)_{n\in\IntegerPP}$.}
  \parameters{$\alpha\in[0.5,1)$, $\lambda\in\RealPP$, $\vect{R}_0$, $\vect{r}_0$,
    $\vect{x}_0$, and $\varpi_0\geq \norm{\vect{R}_0}$.}
  \output{Sequence $(\vect{x}_n)_{n\in\IntegerP}$.}

  \BlankLine
  \initial{%

    $\vect{x}_{1/2} \coloneqq \vect{x}_0 - \alpha\tfrac{1}{\varpi_0}(\vect{R}_0\vect{x}_0 -
    \vect{r}_0)$. \label{HRLSa.step.half}
    
    For any $d\in\Set{1, \ldots, D}$,
    $[\vect{x}_1]_d \coloneqq [\vect{x}_{1/2}]_d \cdot (1 - \lambda / \max\{\lambda, \lvert
    [\vect{x}_{1/2}]_d \rvert\})$. \label{HRLSa.prox.step.1}
    
  }

  \For{$n = 1$ \KwTo $+\infty$}{%

    Set $\varpi_n\geq \norm{\vect{R}_n}$.\label{HRLSa.def.varpi.nplus1}

    $\vect{x}_{n+1/2} \coloneqq \vect{x}_{n} + \vect{x}_{n-1/2} - \vect{x}_{n-1} +
    \alpha\tfrac{1}{\varpi_{n-1}}(\vect{R}_{n-1} \vect{x}_{n-1} - \vect{r}_{n-1}) -
    \tfrac{1}{\varpi_n}(\vect{R}_n \vect{x}_n - \vect{r}_n)$. \label{HRLSa.step.n.half}
    
    For any $d\in\Set{1, \ldots, D}$,
    $[\vect{x}_{n+1}]_d \coloneqq [\vect{x}_{n+1/2}]_d \cdot (1 - \lambda / \max\{\lambda, \lvert
    [\vect{x}_{n+1/2}]_d \rvert\})$. \label{HRLSa.prox.step.n}
    
  }
  \caption{HRLSa}\label{algo:HRLSa}
\end{algorithm}

In the context of \eqref{HLS}, if \eqref{Tn.grad} with $\mu \coloneqq 1$ is adopted, S-FM-HSDM takes
the flavor of \cref{algo:HRLSa}, coined HRLSa. Since $g(\cdot) = \norm{\cdot}_1$,
$\prox_{\lambda g}(\cdot)$ in lines \ref{HRLSa.prox.step.1} and \ref{HRLSa.prox.step.n} of
\cref{algo:HRLSa} boils down to the popular soft-thresholding operation. Following \eqref{Tn.grad},
line~\ref{HRLSa.def.varpi.nplus1} of \cref{algo:HRLSa} introduces an over-estimate $\varpi_n$ of the
maximum eigenvalue $\lambda_{\max}(\vect{R}_n) = \norm{\vect{R}_n}$. To this end, motivated by the
celebrated power iteration~\cite{Golub.VanLoan}, and for an arbitrarily fixed initial vector
$\vect{p}_0\in \mathcal{X}$, the following iterative procedure, run over all $n\in\IntegerPP$, is
used in \cref{Sec:tests} to generate $(\varpi_n)_n$: \textbf{(i)}
$\vect{q}_n \coloneqq \vect{R}_n\vect{p}_{n-1}$; \textbf{(ii)}
$\vect{p}_n \coloneqq \vect{q}_n / \norm{\vect{q}_n}$; \textbf{(iii)}
$\varpi_n \coloneqq \vect{p}_n^{\intercal} \vect{R}_n \vect{p}_n + \epsilon_{\varpi}$, for a
user-defined $\epsilon_{\varpi}\in\RealPP$. If \eqref{Tn.prox} is used as $T_n$ in
\cref{algo:SFMHSDM}, the flavor of S-FM-HSDM is coined HRLSb. Due to space limitations, the detailed
pseudo-code description of HRLSb is omitted. Other options for $T_n$ will be explored
elsewhere. Between HRLSa and HRLSb, HRLSa exhibits the lowest computational complexity, of order
$\mathcal{O}(D^2)$ per $n$. HRLSb requires the matrix inversion
$(\vect{I} + \lambda\vect{R}_n)^{-1}$ for the running average $\vect{R}_n$ in \eqref{Rn.rn}.

\begin{algorithm}[t]
  \DontPrintSemicolon
  \SetKwInOut{input}{Stochastic oracle's input}
  \SetKwInOut{parameters}{User's input}
  \SetKwInOut{output}{Output}
  \SetKwBlock{initial}{Initialization}{}

  \input{$(\vect{a}_n, b_n)_{n\in\IntegerP}$.}

  \parameters{$\alpha\in[0.5,1)$, $\lambda\in\RealPP$, $\vect{R}_0$, $\vect{r}_0$,
    $(\bm{\mathsf{x}}_0^{(1)}, \bm{\mathsf{x}}_0^{(2)})$.}
  
  \output{Sequence $(\vect{x}_n)_{n\in\IntegerP}$.}

  \BlankLine
  \initial{%

    $\vect{x}_0 \coloneqq \tfrac{1}{2}(\bm{\mathsf{x}}_0^{(1)} + \bm{\mathsf{x}}_0^{(2)})$.
    
    $\bm{\mathsf{x}}_{1/2}^{(i)} \coloneqq \alpha\vect{x}_0 + (1-\alpha)
    \bm{\mathsf{x}}_0^{(i)}$, $i\in\Set{1,2}$. \label{CRegLS.step.half}
    
    $\bm{\mathsf{x}}_1^{(1)} \coloneqq (\vect{I} + \lambda\vect{R}_0)^{-1}(\bm{\mathsf{x}}_{1/2}^{(1)} +
    \lambda\vect{r}_0)$. \label{CRegLS.prox.step.(1)}
    
    $[\bm{\mathsf{x}}_1^{(2)}]_d \coloneqq [\bm{\mathsf{x}}_{1/2}^{(2)}]_d \cdot (1 -
    \lambda\rho / \max\{\lambda\rho, \lvert [\bm{\mathsf{x}}_{1/2}^{(2)}]_d
    \rvert\})$, $\forall d\in\Set{1, \ldots, D}$. \label{CRegLS.prox.step.(2)}
    
    $\vect{x}_1 \coloneqq \tfrac{1}{2}(\bm{\mathsf{x}}_1^{(1)} + \bm{\mathsf{x}}_1^{(2)})$.
    
  }

  \For{$n = 1$ \KwTo $+\infty$}{%

    $\bm{\mathsf{x}}_{n+1/2}^{(i)} \coloneqq \bm{\mathsf{x}}_{n-1/2}^{(i)} - \alpha \vect{x}_{n-1} -
    (1-\alpha) \bm{\mathsf{x}}_{n-1}^{(i)} + \vect{x}_n$,
    $i\in\Set{1,2}$. \label{CRegLS.step.n.half}
    
    $\bm{\mathsf{x}}_{n+1}^{(1)} \coloneqq (\vect{I} +
    \lambda\vect{R}_n)^{-1}(\bm{\mathsf{x}}_{n+1/2}^{(1)} +
    \lambda\vect{r}_n)$. \label{CRegLS.prox.step.n.(1)}
    
    $[\bm{\mathsf{x}}_{n+1}^{(2)}]_d \coloneqq [\bm{\mathsf{x}}_{n+1/2}^{(2)}]_d \cdot (1 -
    \lambda\rho / \max\{\lambda\rho, \lvert [\bm{\mathsf{x}}_{n+1/2}^{(2)}]_d \rvert\})$,
    $\forall d\in\Set{1, \ldots, D}$. \label{CRegLS.prox.step.n.(2)}
    
    $\vect{x}_{n+1} \coloneqq \tfrac{1}{2}(\bm{\mathsf{x}}_{n+1}^{(1)} +
    \bm{\mathsf{x}}_{n+1}^{(2)})$.
    
  }
  \caption{S-FM-HSDM(CRegLS)}\label{algo:CRegLS}
\end{algorithm}

In the context of \eqref{CRegLS}, \cref{algo:SFMHSDM} yields \cref{algo:CRegLS}, tagged
S-FM-HSDM(CRegLS). To verify that \cref{algo:CRegLS} is indeed a by-product of \cref{algo:SFMHSDM},
notice that \eqref{CRegLS} can be seen, via variable splitting in the spirit of
\eqref{problem.multiple.h.g}, as
\begin{alignat}{2}
  \min_{(\bm{\mathsf{x}}^{(1)}, \bm{\mathsf{x}}^{(2)}) \in \Real^D\times \Real^D}
  & \Expect && \Bigl[\;\overbrace{\tfrac{1}{2n} \sum\nolimits_{\nu=1}^n
    \left(\vect{a}_{\nu}^{\intercal} \bm{\mathsf{x}}^{(1)} -
      b_{\nu}\right)^2}^{h_n(\bm{\mathsf{x}}^{(1)})}\;\Bigr] \notag\\
  &&& + \underbrace{\rho\norm{\bm{\mathsf{x}}^{(2)}}_1}_{g(\bm{\mathsf{x}}^{(2)})}\ \text{s.to}\
  \bm{\mathsf{x}}^{(1)} = \bm{\mathsf{x}}^{(2)} \,, \label{split.variables.LS.problem}
\end{alignat}
so that space $\mathcal{X}$ is set to be $\Real^D\times \Real^D$, with inner product
$\innerp{(\bm{\mathsf{x}}^{(1)}, \bm{\mathsf{x}}^{(2)})}{(\bm{\mathsf{x}}'^{(1)},
  \bm{\mathsf{x}}'^{(2)})} \coloneqq \innerp{\bm{\mathsf{x}}^{(1)}}{\bm{\mathsf{x}}'^{(1)}} +
\innerp{\bm{\mathsf{x}}^{(2)}}{\bm{\mathsf{x}}'^{(2)}}$. Moreover, $\mathcal{A}$ is the linear
subspace
$\Set{(\bm{\mathsf{x}}^{(1)}, \bm{\mathsf{x}}^{(2)}) \in\mathcal{X} \given \bm{\mathsf{x}}^{(1)} =
  \bm{\mathsf{x}}^{(2)}}$, with (orthogonal) projection mapping given by
$P_{\mathcal{A}}[(\bm{\mathsf{x}}^{(1)}, \bm{\mathsf{x}}^{(2)})] = ( (\bm{\mathsf{x}}^{(1)} +
\bm{\mathsf{x}}^{(2)})/2, (\bm{\mathsf{x}}^{(1)} + \bm{\mathsf{x}}^{(2)})/2)$ and
$T_n \coloneqq T \coloneqq P_{\mathcal{A}}$ in \cref{algo:SFMHSDM}. Moreover,
$\prox_{\lambda(h_n+g)}[(\bm{\mathsf{x}}^{(1)}, \bm{\mathsf{x}}^{(2)})] = (\prox_{\lambda
  h_n}(\bm{\mathsf{x}}^{(1)}), \prox_{\lambda g}(\bm{\mathsf{x}}^{(2)}))$. Lines
\ref{CRegLS.prox.step.n.(1)} and \ref{CRegLS.prox.step.n.(2)} in \cref{algo:CRegLS} correspond to
$\prox_{\lambda h_n}(\cdot)$ and $\prox_{\lambda g}(\cdot)$, respectively. Due to the fact that
$\vect{R}_n$ is obtained by $\vect{R}_{n-1}$ via a rank-one modification, \ie,
$\vect{R}_n = (n-1)\vect{R}_{n-1}/n + \vect{a}_n \vect{a}_n^{\intercal}/n$, a way to compute
$(\vect{I} + \lambda\vect{R}_n)^{-1}$ efficiently via modifications of the eigen-decomposition of
$(\vect{I} + \lambda\vect{R}_{n-1})^{-1}$ can be deduced, for example, via arguments found in
\cite{Gu.Eisenstat.94, Livne.Brandt.02}; details are omitted. Quantities $\vect{R}_0$ and
$\vect{r}_0$ are arbitrarily fixed and used in Line~\ref{CRegLS.prox.step.(1)} of \cref{algo:CRegLS}
to initialize the iterative process.

\subsection{Main theoretical properties}\label{Sec:Main.Properties}

The main properties of S-FM-HSDM (\cref{algo:SFMHSDM}) are summarized in the following
\cref{main.thm,thm:exact.T}. To improve the readability of the manuscript, the detailed description
of the necessary assumptions is deferred to \cref{Sec:under.the.hood}. Nevertheless, as a high-level
description, \cref{as:affine.maps} gathers all the assumptions about the sequence of the
user-defined nonexpansive mappings $(T_n)_n$, such as asymptotic consistency, while \cref{as:losses}
refers to the loss functions $(f_n, h_n)_n$. The typical SA presupposition of asymptotic
unbiasedness is introduced in \cref{as:as.unbiasedness}. The technical \cref{as:domination} imposes
a summability constraint on the random variables (RVs) defined via \eqref{vartheta.start} and
\eqref{BH.thm}. A typical SA boundedness constraint of variances is introduced by
\cref{as:bounded.variances}, while \cref{as:bounded.subgrads} imposes the weak condition on loss
functions that bounded estimates imply bounded subgradients. \cref{as:ergodic.an.bn,as:IID} refer to
the special cases of \eqref{CRegLS} and \eqref{HLS}.

\begin{thm}\label{main.thm}
  Under \cref{as:affine.maps,as:losses,as:as.unbiasedness,as:domination,as:bounded.variances,as:bounded.subgrads}
  (see \cref{Sec:under.the.hood}), the set of cluster points $\mathfrak{C}[(x_n)_n]$ of the
  S-FM-HSDM sequence $(x_n)_n$ (\cref{algo:SFMHSDM}) is nonempty a.s. Furthermore, every point in
  the nonempty $\mathfrak{C}[(x_n)_n]$ is a solution of \eqref{the.problem} a.s.
\end{thm}

\begin{proof}
  See \cref{app:main.thm}.
\end{proof}

\begin{thm}\label{thm:exact.T}
  Consider the case where $T$ is known exactly, \ie, $T=T_n$, $\forall n$. Then, under the same
  setting as in \cref{main.thm}, but without
  \crefnosort{as:affine.maps,as:bounded.var.pi,as:as.bounded.x.as.bounded.xi,as:l2.bounded.x.l2.bounded.xi},
  the sequence $(x_n)_n$ generated by \cref{algo:SFMHSDM} converges a.s.\ to a solution of
  \eqref{the.problem}.
\end{thm}

\begin{proof}
See \cref{app:exact.T}.
\end{proof}

It is worth mentioning here that the qualifier "FM" in S-FM-HSDM comes from the deterministic
predecessor FM-HSDM~\cite{FM-HSDM.Optim.18} and the Fej\'{e}r-monotonicity property of
\eqref{Fejer.cond} and \eqref{Fejer.cond.type3} in \cref{app:main.thm}.

Since HRLS and S-FM-HSDM(CRegLS) are offsprings of S-FM-HSDM, assertions about their convergence
properties can be deduced from \cref{main.thm,thm:exact.T} and can take various forms. This study
avoids to provide an exhaustive list of all such assertions with their forms, but brings only a
couple of examples in the form of the following corollaries.

\begin{coroll}\label{cor:HRLS} Let \crefnosort{as:ergodic.an.bn,as:domination,as:bounded.var.pi} hold
  true. Assume also that the stochastic process $(\vect{a}_n)_n$ possesses a non-singular
  $\vect{R}$, and that $\exists \varpi\in \RealPP$ s.t.\
  $\varpi_n \coloneqq \varpi \geq \max\Set{\norm{\vect{R}}, \norm{\vect{R}_n}}$, a.s., $\forall
  n$. Then, the set of cluster points of the sequence $(\vect{x}_n)_n$, generated by either HRLSa or
  HRLSb, is non-empty, and any of its cluster points is a solution of \eqref{HLS} a.s.
\end{coroll}

\begin{proof}
See \cref{app:HRLS}.
\end{proof}

As a postscript to \cref{cor:HRLS}, recall that the matrix $\vect{R}$ of \textit{any}\/ regular
process $(\vect{a}_n)_n$, \ie, a process with non-zero innovation~\cite[\S2.6]{Porat.book}, is
non-singular~\cite[Prob.~2.2]{Porat.book}.

\begin{coroll}\label{cor:S-FM-HSDM(CRegLS)} Let
  \crefnosort{as:ergodic.an.bn,as:unbiased.hn,as:domination} hold true. Then, the sequence
  $(\vect{x}_n)_n$ generated by \cref{algo:CRegLS} converges a.s.\ to a solution of \eqref{CRegLS}.
\end{coroll} 

\begin{proof}
See \cref{app:S-FM-HSDM(CRegLS)}.
\end{proof}

\section{Performance Analysis of S-FM-HSDM}\label{Sec:under.the.hood}

Rather than simply listing all assumptions needed for \cref{main.thm,thm:exact.T}, as well as for
\cref{cor:HRLS,cor:S-FM-HSDM(CRegLS)}, this section follows a more instructive route by exemplifying
the assumptions in the context of \eqref{CRegLS} and \eqref{HLS}. With symbol $\limas_n$ introduced
in \cref{app:preliminaries}, the following assumptions are imposed on the mappings $T$ and
$(T_n)_n$.

\begin{assumption}[Mappings $T$ and $T_n$]\label{as:affine.maps}\mbox{}
  \begin{asslist}
  \item\label[assumption]{as:T} $T\in \mathfrak{T}_{\mathcal{A}}$.
  \item\label[assumption]{as:Tn} $T_n \coloneqq Q_n + \pi_n$, where mapping
    $Q_n: \mathcal{X} \to \mathcal{X}$ is positive, with $\norm{Q_n} \leq 1$, and
    $\pi_n\in \mathcal{X}$, a.s., $\forall n$.
  \item\label[assumption]{as:consistent.Tn} $(T-T_n) \limas_n 0$, \ie, $(T-T_n)x \limas_n 0$, $\forall
    x\in\mathcal{X}$, or, equivalently, $(Q-Q_n) \limas_n 0$ and $(\pi-\pi_n) \limas_n 0$.
  \item\label[assumption]{as:unbiased.Tn} Define $\forall n$,
    $t_n \coloneqq \Expect_{\given\mathcal{F}_n} \left[ \sum\nolimits_{\nu=1}^n (T-T_{\nu})x_{\nu}
    \right]$. All cluster points of any bounded subsequence of $(\Expect(t_n))_n$ belong to
    $\range(\Id-Q)$. \qedhere
  \end{asslist}
\end{assumption}

\noindent To underline the generality of \cref{as:affine.maps}, the following popular
\cref{as:ergodic.an.bn}, placed in the context of \cref{Sec:system.id} and \eqref{Rn.rn}, provides a
special case of \cref{as:affine.maps}, as \cref{lem:Tn.LS} demonstrates.

\begin{assumption}[Pointwise ergodicity]\label{as:ergodic.an.bn}
  $\bm{\mathcal{E}}^R_n \coloneqq \vect{R}-\vect{R}_n \limas_n \vect{0}$ and
  $\bm{\varepsilon}^r_n \coloneqq \vect{r}-\vect{r}_n \limas_n \vect{0}$.
\end{assumption}

\noindent To save space, a discussion on conditions which suffice to guarantee
\cref{as:ergodic.an.bn}, such as statistical independency or mixing conditions~\cite{Andrews.88,
  Petersen}, via the strong law of large numbers, is omitted. Notice that due to
\cref{as:ergodic.an.bn}, $(\vect{R}_n)_n$ is bounded a.s.; hence,
$\exists \varpi \coloneqq \varpi(\omega) \geq \max \Set{\norm{\vect{R}}, \norm{\vect{R}_n}}$, a.s.,
$\forall n$ (symbol $\omega$ is introduced in \cref{app:preliminaries}).

\begin{lemma}\label{lem:Tn.LS} Assume that $\exists \varpi\in\RealPP$ s.t.\
  $\varpi_n \coloneqq \varpi \geq \max \Set{\norm{\vect{R}}, \norm{\vect{R}_n}}$, a.s., $\forall n$,
  and that the matrix $\vect{R}$ of the stochastic process $(\vect{a}_n)_n$ is non-singular. Then,
  under also \cref{as:ergodic.an.bn}, mappings \eqref{Tn.LS} satisfy
  \cref{as:Tn,as:consistent.Tn}. Moreover, for any $T\in \mathfrak{T}_{\mathcal{A}}$ and any of its
  estimates $(T_n)_n$, \cref{as:unbiased.Tn} holds true. \qedhere

\end{lemma}

\begin{proof}
  See \cref{app:Tn.LS}.
\end{proof}

\begin{assumption}[Loss functions]\label{as:losses}\mbox{}
  \begin{asslist}
    
  \item\label[assumption]{as:convex.losses} $f, h, g: \mathcal{X}\to \Real\cup \Set{+\infty}$ belong
    to the class $\Gamma_0(\mathcal{X})$ of proper, lower semicontinuous (l.s.c.), convex
    functions~\cite{HB.PLC.book}.
    
  \item\label[assumption]{as:f.Lipschitz} $f$ is everywhere (Fr\'{e}chet) differentiable, with
    $L_{\nabla f}$-Lipschitz continuous $\nabla f$:
    $\norm{\nabla f(x) - \nabla f(x')} \leq L_{\nabla f} \norm{x-x'}$,
    $\forall (x,x')\in \mathcal{X} \times \mathcal{X}$, for some $L_{\nabla f}\in\RealPP$. Moreover,
    for any sub-$\sigma$-algebra $\mathcal{G}$ of $\Sigma$ (\cf \cref{app:preliminaries}) and
    $\forall x\in \mathrm{m}\mathcal{G}$, $\nabla f(x)\in \mathrm{m}\mathcal{G}$.
    
  \item\label[assumption]{as:convex.losses.n} $f_n, h_n\in\Gamma_0(\mathcal{X})$ a.s., $\forall n$.
    
  \item\label[assumption]{as:fn.Lipschitz} $f_n$ is everywhere (Fr\'{e}chet) differentiable, with
    $L_n$-Lipschitz continuous $\nabla f_n$ a.s., $\forall n$.
    
  \item\label[assumption]{as:bounded.Lips} There exist $n_{\mypound} \in\IntegerP$ and a
    $C_{\text{Lip}}\in \RealPP$, which is constant over all $\omega\in\Omega$, s.t.\
    $L_n \leq C_{\text{Lip}}$ a.s., $\forall n\geq n_{\mypound}$.
    
  \item\label[assumption]{as:consistent.gradfn} $(\nabla f - \nabla f_n) \limas_n 0$. \qedhere
    
  \end{asslist}
\end{assumption}

\noindent It can be readily verified in the context of \eqref{CRegLS} that
\cref{as:convex.losses,as:f.Lipschitz,as:convex.losses.n,as:fn.Lipschitz} are satisfied by either
$(f, f_n) \coloneqq (l, l_n)$, or, $(h, h_n) \coloneqq (l, l_n)$. If $f_n \coloneqq l_n$, then
$\forall \vect{x}\in \mathcal{X}$,
\begin{align}
  \nabla f_n(\vect{x})
  & = \tfrac{1}{n} \sum\nolimits_{\nu=1}^n \vect{a}_{\nu}
    \left( \vect{a}_{\nu}^{\intercal} \vect{x} - b_{\nu} \right) = \vect{R}_n \vect{x} -
    \vect{r}_n \,. \label{grad.fn.LS}
\end{align}
Scalar $L_n \coloneqq \norm{\vect{R}_n}$ can be considered as $\nabla f_n$'s Lipschitz
coefficient. Hence, if $(\vect{R}_n)_n$ is uniformly bounded over $\Omega$, \ie,
$\exists C_{\text{Lip}}$ s.t.\ $\norm{\vect{R}_n} \leq C_{\text{Lip}}$, a.s., $\forall n$, then
\cref{as:bounded.Lips} holds true. On the other hand, if the uniform boundedness of $(\vect{R}_n)_n$
cannot be guaranteed, and since the current framework places no requirements on the uniform
boundedness of the subgradients of $(h_n)_n$, \eqref{the.problem} offers the flexibility to set
$h_n \coloneqq l_n$ and $f_n \coloneqq 0$, for which \cref{as:bounded.Lips} holds trivially
true. Given that $\nabla f (\vect{x}) = \vect{Rx} - \vect{r}$, $\forall\vect{x} \in \mathcal{X}$,
whenever $f = l$, it can be verified via \cref{as:ergodic.an.bn} and \eqref{grad.fn.LS} that
\cref{as:consistent.gradfn} holds true.

\begin{assumption}[Asymptotic unbiasedness]\label{as:as.unbiasedness}\mbox{}
  \begin{asslist}
  \item\label[assumption]{as:unbiased.hn} For any $x\in\mathcal{X}$,
    $\Expect_{\given \mathcal{F}_n} [(h-h_{n})(x)] \eqqcolon \varepsilon_n^h(x) \limas_n 0$ and
    $\varepsilon_n^h(x_n) \limas_n 0$.
  \item\label[assumption]{as:unbiased.gradfn}
    $\Expect_{\given \mathcal{F}_n} [(\nabla f - \nabla f_n)(x_n)] \eqqcolon \varepsilon_n^f(x_n)
    \limas_n 0$. \qedhere
  \end{asslist}
\end{assumption}

\noindent Asymptotic unbiasedness appears often in SA, \eg, \cite[p.~132,
Thm.~2.3]{Kushner.Yin}. \cref{lem:unbiasedness.LS} presents cases where \cref{as:as.unbiasedness}
holds true. Several of the results of \cref{lem:unbiasedness.LS} serve also as intermediate steps
that justify the introduction of \cref{as:domination}; more precisely, \eqref{strict.conditions}
suffice for \cref{as:domination} to hold true. To prove the claims of \cref{lem:unbiasedness.LS},
the following assumption is needed.

\begin{assumption}\label{as:IID}
  Motivated by \cite[p.~162]{Kushner.Yin}, the approximation errors $\bm{\mathcal{E}}_n^R$ and
  $\bm{\varepsilon}_n^r$ (\cf~\cref{as:ergodic.an.bn}) are assumed here to be \textit{exogenous}\/
  w.r.t.\ $\mathcal{F}_n \coloneqq \sigma(\Set{x_{\nu}}_{\nu=0}^n)$ for any $n$. In other words,
  $\bm{\mathcal{E}}_n^R$, $\bm{\varepsilon}_n^r$, which are provided by the stochastic oracle, are
  considered to be independent of the past history $\mathcal{F}_n$ of the iterates: $\forall n$ and
  a.s., $\Expect_{\given\mathcal{F}_n} (\bm{\mathcal{E}}_n^R) = \Expect (\bm{\mathcal{E}}_n^R)$,
  $\Expect_{\given\mathcal{F}_n} (\bm{\varepsilon}_n^r) = \Expect (\bm{\varepsilon}_n^r)$.
  Moreover, the stochastic process $(\vect{a}_n)_n$ is assumed to be independent and identically
  distributed (IID) and $\mathcal{F}_n$ is considered to be conditionally independent with
  $\sigma(\Set{\vect{a}_{\nu}}_{\nu=1}^n)$ given $\sigma(\vect{R}_n)$.
\end{assumption}

It can be verified via the stationarity conditions of \cref{Sec:system.id} that \cref{as:IID}
implies $\Expect_{\given\mathcal{F}_n} (\vect{R}_n) = \Expect (\vect{R}_n) = \vect{R}$ and
$\Expect_{\given\mathcal{F}_n} (\vect{r}_n) = \Expect (\vect{r}_n) = \vect{r}$. Thus,
$\Expect_{\given\mathcal{F}_n} (\bm{\mathcal{E}}_n^R) = \vect{0}$ and
$\Expect_{\given\mathcal{F}_n} (\bm{\varepsilon}_n^r) = \vect{0}$.

\begin{lemma}\label{lem:unbiasedness.LS}\mbox{}

  \begin{lemlist}

  \item\label[lemma]{lem:results.Tn} Consider $T$ of \eqref{T.grad}, $T_n$ of \eqref{Tn.grad}, and let
    \cref{as:IID} hold true. Moreover, assume the existence of $\varpi \in \RealPP$ s.t.\
    $\varpi_n \coloneqq \varpi \geq \max\Set{\norm{\vect{R}}, \norm{\vect{R}_n}}$, a.s.\ and
    $\forall n$. Then,
    \begin{subequations}\label{strict.conditions}
      \begin{align}
        \Expect_{\given \mathcal{F}_n}[(T-T_n) \vect{x}_n]
        & = \vect{0} \,, \label{unbiased.Tn} \\
        \Expect_{\given \mathcal{F}_n}[(Q-Q_n) (\vect{x}_n- \vect{x}_{n-1})]
        & = \vect{0} \,, \label{unbiased.Qn} \\ 
        \vect{t}_n
        & = \vect{0} \,, \label{tn.equals.zero} \\
        \Expect_{\given \mathcal{F}_n}[(T_n-T_{n-1}) \vect{x}_{n-1}]
        & = \vect{0} \,. \label{increments.Tn} 
      \end{align}
      Clearly, $\Expect(\vect{t}_n) = \vect{0}$, and thus, the only cluster point
      $\lim\nolimits_{n\to\infty} \Expect(\vect{t}_n) = \vect{0}$ of sequence $(\vect{t}_n)_n$
      belongs trivially to $\range (\Id-Q)$; that is, \cref{as:unbiased.Tn} holds true.

    \item\label[lemma]{lem:hn.CRegLS.as.unbiasedness} Consider $(h, h_n) \coloneqq (l, l_n)$ in
      \eqref{CRegLS}, and let \cref{as:IID} hold true. Then, \cref{as:unbiased.hn} holds true with
      $\varepsilon_n^h(\vect{x}) = 0 = \varepsilon_n^h(\vect{x}_n)$, a.s., $\forall n$,
      $\forall\vect{x}\in \mathcal{X}$.

    \item Consider $\nabla f_n$ in \eqref{grad.fn.LS}, $f \coloneqq l$, and let \cref{as:IID} hold
      true. Then, $\forall n$,
      \begin{align}
        \Expect_{\given \mathcal{F}_n}[(\nabla f - \nabla f_n) \vect{x}_n]
        & = \vect{0} \,, \label{unbiased.gradfn} \\
        \Expect_{\given \mathcal{F}_n}[(\nabla f_n - \nabla f_{n-1}) \vect{x}_{n-1}]
        & = \vect{0} \,. \label{increments.gradfn} 
      \end{align}

    \item\label[lemma]{as:subgrad.CRegLS} Let \cref{as:IID} hold true. Let also
      $g \coloneqq \norm{\cdot}_1$ and either $(h, h_n) \coloneqq (0, 0)$ or
      $(h, h_n) \coloneqq (l, l_n)$. Consider also the sequence
      $(\bm{\xi}_n \in \partial(h_{n-1}+g)(\vect{x}_n))_n$ of subgradients defined in
      \eqref{existence.xi}. Then,
      $\Expect_{\given\mathcal{F}_n} (\bm{\xi}_n) \in\partial (h+g)(\vect{x}_n)$, $\forall n$. More
      generally,
      \begin{align}
      & \exists (\epsilon_n)_n \subset (\mathrm{m}\Sigma)^+\ \text{with}\
        \sum\nolimits_n \Expect (\epsilon_n) < +\infty \notag\\
      & \text{s.t.}\ \Expect_{\given\mathcal{F}_n} (\bm{\xi}_n) \in\partial_{\epsilon_n}
        (h+g)(\vect{x}_n)\,, \forall n\,. \label{unbiased.subgrad.hg}
      \end{align}
    \end{subequations}
  \end{lemlist}
\end{lemma}

\begin{proof}
  See \cref{app:unbiasedness.LS}.
\end{proof}

\noindent Results \eqref{strict.conditions} can be relaxed as follows: \cref{app:main.thm}
[\cf~\eqref{prove.suff.cond.theta}] demonstrates that \eqref{strict.conditions} suffice to establish
\cref{as:domination}.

\begin{assumption}[Dominated $(\vartheta_n)_{n\in\IntegerP}$]\label{as:domination}
  Consider the sequence $(\vartheta_n)_{n\in\IntegerP}$ of RVs defined by the expression which
  starts from \eqref{vartheta.start} and ends at \eqref{BH.thm}. There exists
  $\psi\in (\mathrm{m}\Sigma)^+$ with $\Expect(\psi) < +\infty$ s.t.
  $\sum\nolimits_n \Expect_{\given \mathcal{F}_n} (\vartheta_n)^+ \leq \psi$ a.s., where
  $\Expect_{\given \mathcal{F}_n} (\vartheta_n)^+ \coloneqq \max\Set{0, \Expect_{\given
      \mathcal{F}_n} (\vartheta_n)}$.
\end{assumption}

\begin{assumption}[Bounded variances]\label{as:bounded.variances}\mbox{}
  \begin{asslist}
  \item\label[assumption]{as:bounded.var.gradfn} Given $z\in\mathcal{X}$, there exists
    $C_{\nabla f} \coloneqq C_{\nabla f}(z)\in \RealPP$ s.t.\
    $\Expect [\norm{(\nabla f - \nabla f_n)z}^2] \leq C_{\nabla f}$, $\forall n$.
  \item\label[assumption]{as:bounded.var.pi} There exists $C_{\pi}\in\RealPP$ s.t.\
    $\Expect (\norm{\pi_n-\pi}^2) \leq C_{\pi}$, $\forall n$. \qedhere
  \end{asslist}
\end{assumption}

\noindent Bounded-variance assumptions appear often in SA, \eg, \cite[p.~126, (A2.1)]{Kushner.Yin}.

\begin{assumption}[Bounded estimates yield bounded subgradients]\label{as:bounded.subgrads}
  \mbox{}
  \begin{asslist}
  \item\label[assumption]{as:bounded.tau} For any a.s.\ bounded $(z_n)_n$, there exist a sequence
    $(\tau_n)_n$ and $C_{\partial} \coloneqq C_{\partial}(\omega) \in\RealPP$ s.t.\
    $\tau_n\in \partial (h_n+g)(z_n)$ and
    $\Expect_{\given\mathcal{F}_n} (\norm{\tau_n}) \leq C_{\partial}$, $\forall n$, a.s.
  \item\label[assumption]{as:as.bounded.x.as.bounded.xi} Consider the sequence $(\xi_n)_n$ of
    subgradients defined in \eqref{existence.xi}. If $(x_n)_n$ is bounded a.s., then $(\xi_n)_n$ is
    bounded a.s.
  \item\label[assumption]{as:l2.bounded.x.l2.bounded.xi} If $(\Expect(\norm{x_n}^2))_n$ is bounded,
    then $(\Expect(\norm{\xi_n}^2))_n$ is bounded. \qedhere
  \end{asslist}
\end{assumption}

\begin{lemma}\label{lem:bounded.subgrads} Let \cref{as:ergodic.an.bn} hold
  true. If $(h, h_n) \coloneqq (0, 0)$ and $g$ is defined as a scalar multiple of $\norm{\cdot}_1$,
  then \cref{as:bounded.subgrads} holds true. If $(h, h_n) \coloneqq (l, l_n)$ and
  $(f, f_n) \coloneqq (0, 0)$ in \eqref{CRegLS}, then the following claims can be established.
  \begin{lemlist}
  \item\label[lemma]{lem:bounded.tau.as.bounded.x.as.bounded.xi}
    \cref{as:bounded.tau,as:as.bounded.x.as.bounded.xi} hold true.
  \item If there exist also $\varpi, \varpi'\in \RealPP$, fixed over the probability space, s.t.\
    $\norm{\vect{R}_n}\leq \varpi$ and $\norm{\vect{r}_n}\leq \varpi'$, $\forall n$ and a.s., then
    \cref{as:l2.bounded.x.l2.bounded.xi} holds true. \qedhere
  \end{lemlist}
\end{lemma}

\begin{proof}
  See \cref{app:bounded.subgrads}.
\end{proof}

\section{Numerical Tests}\label{Sec:tests}

The proposed framework is validated within the setting of \cref{Sec:system.id} where
S-FM-HSDM(CRegLS), HRLSa and HRLSb are compared with the following OL and SA schemes:
\begin{enumerate}
\item The classical RLS~\cite[\S30.2]{SayedBook};
\item the $\ell_1$-norm regularized ($\ell_1$-)RLS~\cite{l0RLS}, and its extension, the
  $\ell_0$-norm ($\ell_0$-)RLS~\cite{l0RLS}, where a \textit{non-convex}\/
  regularizing function is used instead of $\norm{\cdot}_1$;
\item the LASSO-motivated online selective coordinate descent (OSCD) and online cyclic coordinate
  descent (OCCD) methods~\cite{RLS.meets.l1}, where, according to \cite[Sec.~V]{RLS.meets.l1}, the
  power of the additive noise in the linear-regression model is assumed to be known and incorporated
  in the regularizing coefficient $\rho_n$ in \eqref{CRegLS} s.t.\ $\rho_n\to_n 0$;
\item the proximal stochastic variance-reduced gradient (Prox-SVRG) method~\cite{ProxSVRG}, applied
  to the setting of the ever-growing data regime
  $f \coloneqq (1/n) \sum_{\nu=1}^n \mathfrak{f}_{\nu}$ in \eqref{CRegLS}, with
  $\mathfrak{f}_{\nu}(\vect{x}) \coloneqq (1/2)(\vect{a}_{\nu}^{\intercal} \vect{x} - b_{\nu})^2$;
\item SVRG-ADMM~\cite{FastLightSADMM}, where $f$ is identical to that of the Prox-SVRG case;
\item the accelerated stochastic approximation (ACSA) with the step sizes
  of~\cite[(33)]{Lan.ACSA.12};
\item the adaptive sparse variational Bayes multi-parameter Laplace prior
  (ASVB-MPL) method~\cite{Themelis.ASVB.14}; and
\item the stochastic dual-averaging (SDA) scheme with linear-convergence-rate
  guarantees~\cite{Flammarion.Bach.17a}.
\end{enumerate}
It is worth stressing here that all of \cite{l0RLS, RLS.meets.l1, ProxSVRG, FastLightSADMM,
  Lan.ACSA.12, Flammarion.Bach.17a} are built around the mainstream \eqref{CRegLS}. As explained in
\cref{Sec:problem,Sec:system.id}, any attempt to pass $\mathcal{A}$ of \eqref{HLS} to the objective
function via the indicator function $\iota_{\mathcal{A}}$ entails the use of the projection mapping
$P_{\mathcal{A}}$ and, thus, the eigen-decompositions of $(\vect{R}_n)_n$ via the
(Moore-Penrose-)pseudoinverse operation. Recall that this is not the case for the computationally
"light" HRLSa.

In all tests, the dimension of the Euclidean space $\mathcal{X} = \Real^D$ is set to be
$D \coloneqq 100$. The sparse system $\bm{\theta}_*$ is created by placing $\pm 1$s at randomly
selected entries of the $D\times 1$ all-zero vector. The ``sparsity level'' of $\bm{\theta}_*$ is
defined as the percentage of the number of non-zero entries of $\bm{\theta}_*$ over $D$. All of the
methods were tested in several scenarios detailed below. Since focus is placed on the
system-identification problem of \cref{Sec:system.id}, the criterion of performance is the
normalized-root-mean-square-deviation loss
$\norm{\vect{x}_n - \bm{\theta}_*}/ \norm{\bm{\theta}_*}$. Each curve in the figures is the uniform
average of $500$ independently performed tests.

To ensure fair comparisons, the parameters of every method were carefully tuned to reach optimal
performance per given scenario. Due to space limitations, lists of all parameters for each test are
omitted. However, few things can be stated here about the parameters $\alpha$ and $\lambda$ of
\cref{algo:SFMHSDM}. With $\alpha\in [0.5,1)$ in \cref{algo:SFMHSDM}, the general trend is that the
fastest convergence speed of S-FM-HSDM is achieved for $\alpha = 0.5$. Moreover, with
$\lambda \in (0, 2(1-\alpha)/L_{\nabla f})$, the fastest convergence speed was observed for values
of $\lambda$ close to $2(1-\alpha)/L_{\nabla f}$. In the case where $L_{\nabla f}$ is unknown, \eg,
the case of $f \coloneqq 0$, the values of $L_{\nabla f}$ used in the following tests were drawn
from the interval $[10^{-3}, 10^{-1}]$.

\subsection{$(\vect{a}_n)_n$ is an IID process}\label{Sec:an.IID}

With regards to the linear-regression model of \cref{Sec:system.id}, process $(\vect{a}_n)_n$ is
considered to be IID Gaussian. Independency is also assumed among the entries
$([\vect{a}_n]_d)_{d=1}^D$ of each vector $\vect{a}_n$, $\forall n$. Given a value for the
signal-to-noise ratio (SNR) in dB, the ``power'' of the additive noise
$\Expect(\eta_n^2) = 10^{-\text{SNR(dB)/10}} \norm{\bm{\theta}_*}^2 \Expect([\vect{a}_n]_d^2)$. The
SNR values $\Set{10, 20}\text{dB}$ were examined and results are illustrated in
\cref{fig:Sce1-2,fig:Sce3-4}. Remarkably, the \eqref{HLS} formulation seems to be more appropriate
than \eqref{CRegLS} for the sparse system-identification problem: The best performance among all
methods is achieved by the proposed HRLSa, HRLSb and the non-convex $\ell_0$-RLS.

\subsubsection{Time-varying system} To test the ability of the methods to adapt to dynamic system
changes, a typical AF test is considered here~\cite{SayedBook}: The sparsity level of the estimandum
$\bm{\theta}_*$ changes abruptly at the time instance $2,500$ from $1\%$ to $10\%$, where the
non-zero entries of $\bm{\theta}_*$ are re-allocated randomly.

As in the classical exponentially-weighted RLS~\cite[\S30.6]{SayedBook}, \eqref{CRegLS} is modified
to
\begin{align*}
  \min_{\vect{x}\in\Real^D} \Expect \Bigl[\tfrac{1}{2\Gamma_{\text{f},n}}
  \sum\nolimits_{\nu=1}^n \gamma_{\text{f}}^{n-\nu} \left(\vect{a}_{\nu}^{\intercal} \vect{x} -
  b_{\nu}\right)^2 \Bigr] + \rho\norm{\vect{x}}_1 \,,
\end{align*}
where $\Gamma_{\text{f},n} \coloneqq \sum_{\nu=1}^n \gamma_{\text{f}}^{n-\nu}$ and
$\gamma_{\text{f}}\in (0,1]$ is a ``forgetting coefficient'' that enforces a ``short-memory''
effect, via the exponential rule $\gamma_{\text{f}}^{n-\nu}$, to account for the non-stationaries of
the input-output data statistics. Results are illustrated in \cref{fig:Sce11}. HRLSa, HRLSb and the
Bayesian ASVB-MPL seem to be both agile and accurate in their estimation task.

\subsubsection{No additive noise} Here, $\eta_n = 0$, or, $\text{SNR} = +\infty$, in the
linear-regression model of \cref{Sec:system.id}. Results are illustrated in \cref{fig:Sce9-10}. The
best performance is achieved by S-FM-HSDM(CRegLS), HRLSa, HRLSb, SDA, Prox-SVRG and SVRG-ADMM.

\subsection{$(\vect{a}_n)_n$ is an auto-regressive (AR) process}\label{Sec:an.AR}

A first-order auto-regressive [AR($1$)] process $(\vect{a}_n)_n$ is considered:
$\forall n\in\IntegerP$, $\vect{a}_n \coloneqq \delta_{\text{AR}}\vect{a}_{n-1} + \bm{\upsilon}_n$,
with $\delta_{\text{AR}}\in \Real$ and $\lvert \delta_{\text{AR}}\rvert < 1$,
$(\bm{\upsilon}_n)_{n=-1}^{+\infty}$ is a zero-mean Gaussian IID process, where independency is also
assumed among the entries $([\bm{\upsilon_n}]_d)_{d=1}^D$ of each vector $\bm{\upsilon}_n$, and
$\vect{a}_{-1} \coloneqq \bm{\upsilon}_{-1}$. Recall here that
$\Expect([\vect{a}_n]_d^2) = \Expect([\bm{\upsilon}_n]_d^2)/
(1-\delta_{\text{AR}}^2)$~\cite[(2.12.7)]{Porat.book}. In all tests, the ratio
$\Expect([\vect{a}_n]_d^2)/ \Expect([\bm{\upsilon}_n]_d^2)$ is set to $5$dB. Results are illustrated
in \cref{fig:Sce5-6,fig:Sce7-8}. The best performance is achieved by HRLSa, HRLSb and $\ell_0$-RLS.

\begin{figure}	
  \centering
  \begin{subfigure}[t]{.45\linewidth}
    \centering
    \includegraphics[width=\linewidth]{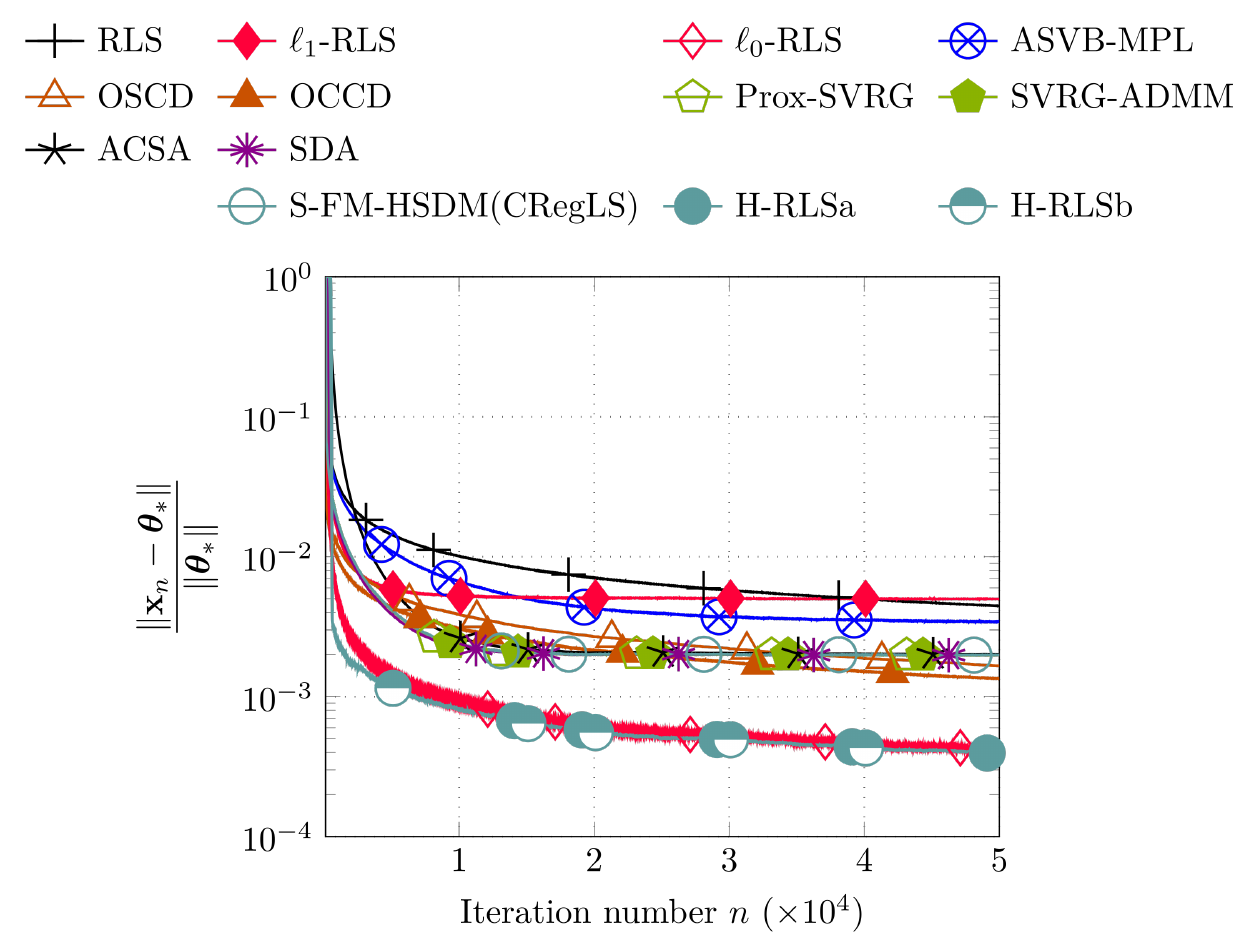}
    \subcaption{Sparsity level: $1\%$.}\label{fig:Sce1}
  \end{subfigure}
  \begin{subfigure}[t]{.45\linewidth}
    \centering
    \includegraphics[width=\linewidth]{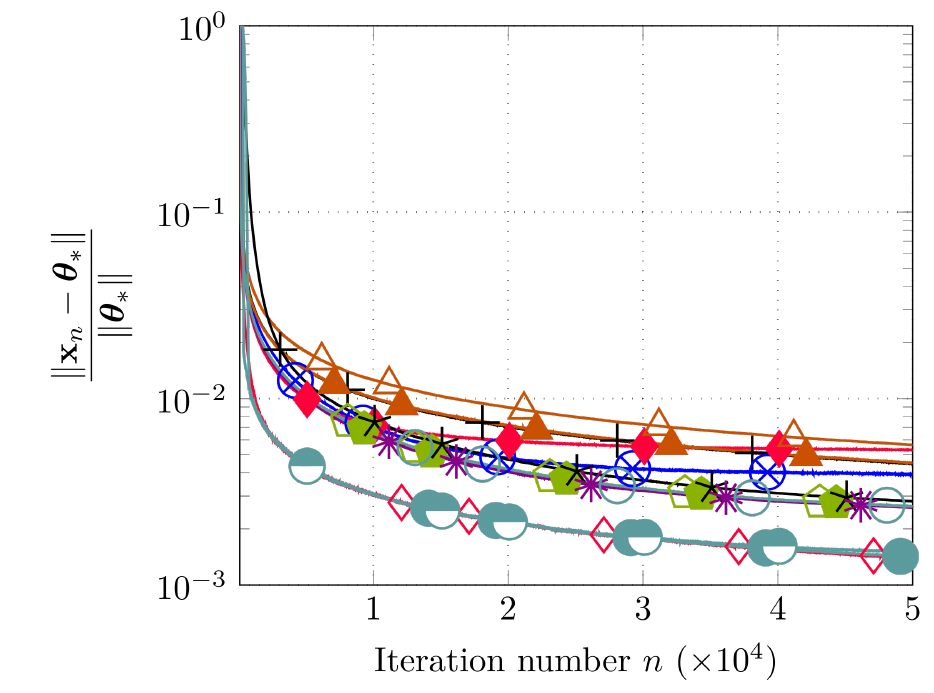}
    \subcaption{Sparsity level: $10\%$.}\label{fig:Sce2}
  \end{subfigure}
  \caption{IID $(\vect{a}_n)_n$; $\text{SNR} = 20\text{dB}$.}\label{fig:Sce1-2}
\end{figure}

\begin{figure}	
  \centering
  \begin{subfigure}[t]{.45\linewidth}
    \centering
    \includegraphics[width=\linewidth]{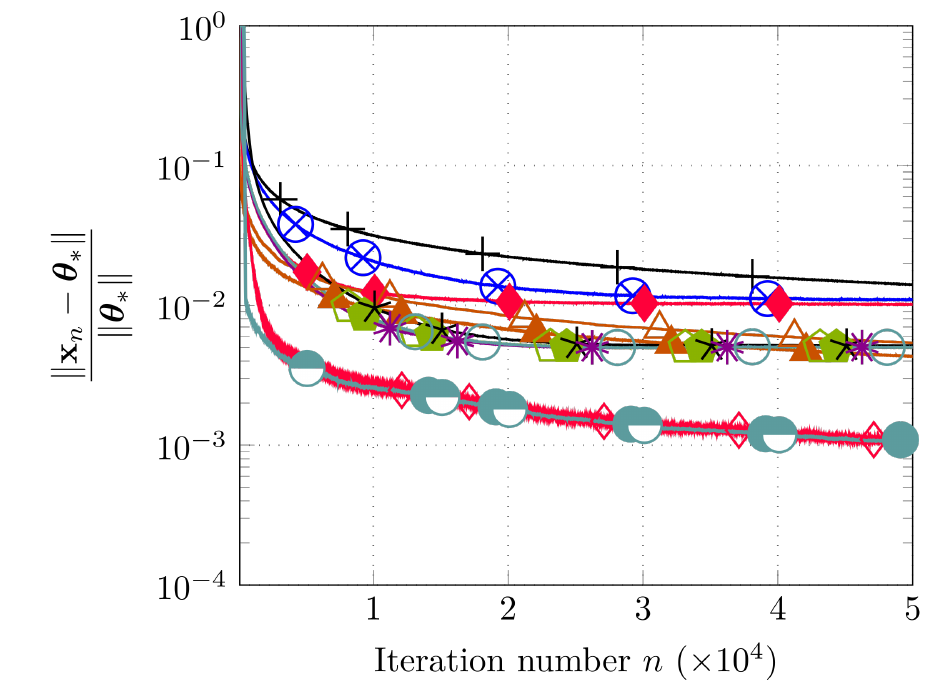}
    \subcaption{Sparsity level: $1\%$.}\label{fig:Sce3}
  \end{subfigure}
  \begin{subfigure}[t]{.45\linewidth}
    \centering
    \includegraphics[width=\linewidth]{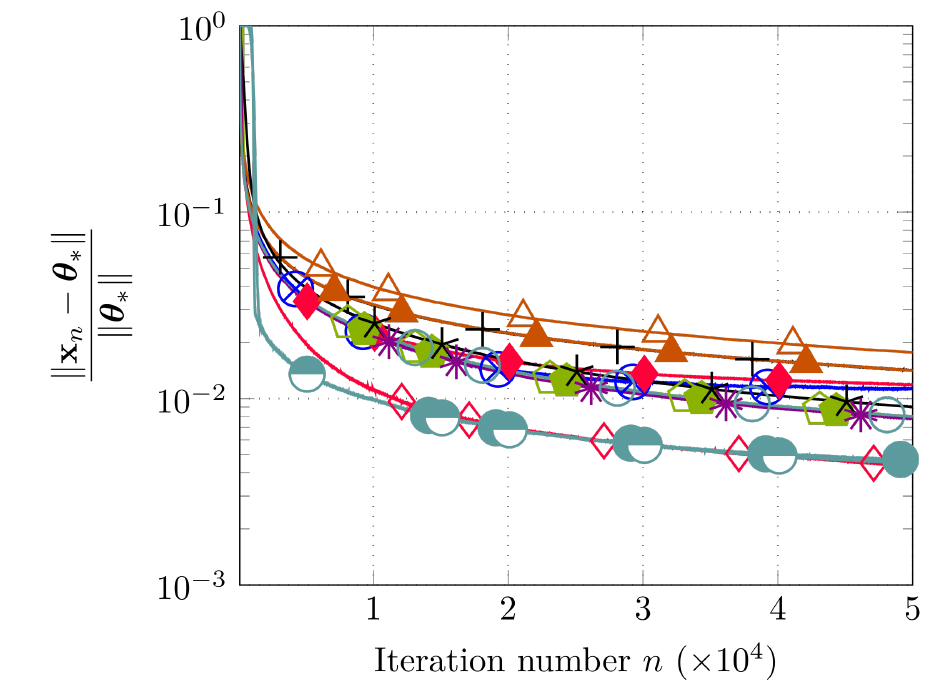}
    \subcaption{Sparsity level: $10\%$.}\label{fig:Sce4}
  \end{subfigure}
  \caption{IID $(\vect{a}_n)_n$; $\text{SNR} = 10\text{dB}$.}\label{fig:Sce3-4}
\end{figure}

\begin{figure}[!t]
  \centering
  \includegraphics[width=.5\linewidth]{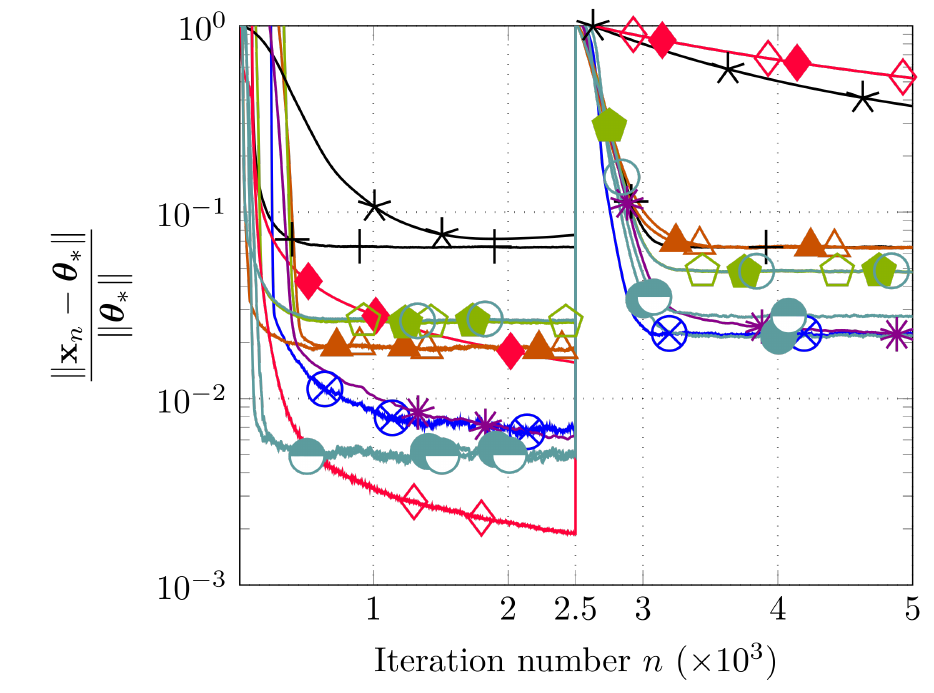}
  \caption[]{IID $(\vect{a}_n)_n$; $\text{SNR} = 20\text{dB}$; the sparsity level of $\bm{\theta}_*$
    changes at time $n = 2,500$ from $1\%$ to $10\%$.}\label{fig:Sce11}
\end{figure}

\begin{figure}	
  \centering
  \begin{subfigure}[t]{.45\linewidth}
    \centering
    \includegraphics[width=\linewidth]{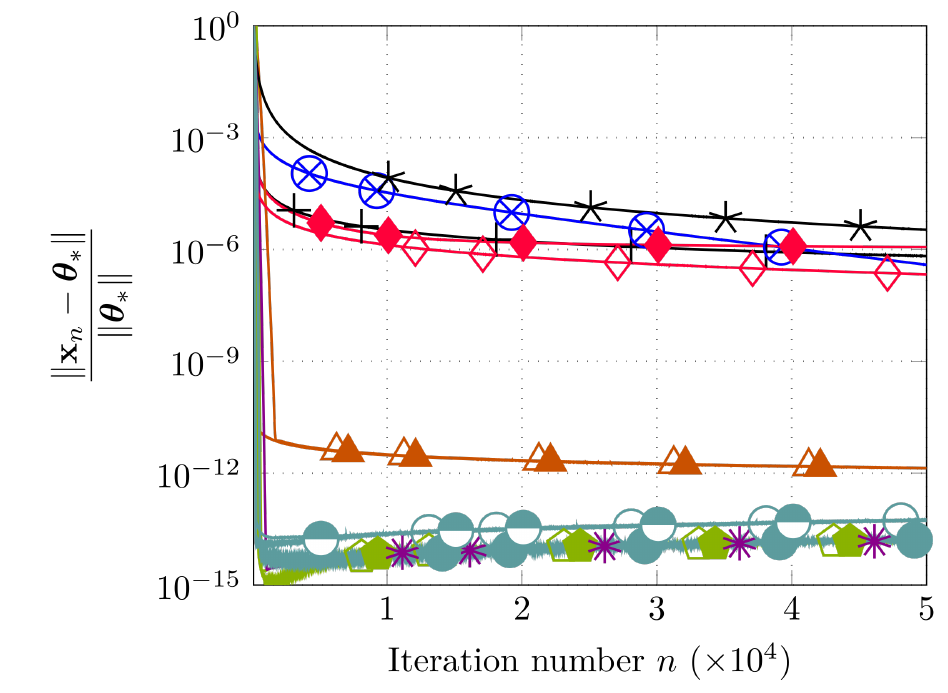}
    \subcaption{IID $(\vect{a}_n)_n$.}\label{fig:Sce9}
  \end{subfigure}
  \begin{subfigure}[t]{.45\linewidth}
    \centering
    \includegraphics[width=\linewidth]{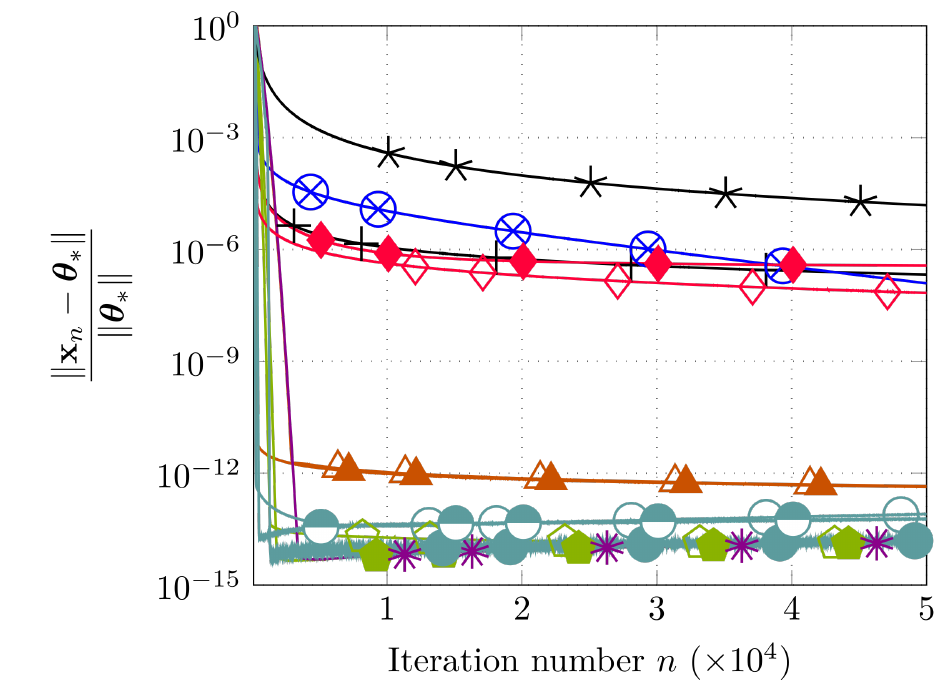}
    \subcaption{AR $(\vect{a}_n)_n$.}\label{fig:Sce10}
  \end{subfigure}
  \caption{$\text{SNR} = +\infty$ (no additive noise); sparsity level: $10\%$; $\rho =
    10^{-20}$. }\label{fig:Sce9-10}
\end{figure} 

\begin{figure}	
  \centering
  \begin{subfigure}[t]{.45\linewidth}
    \centering
    \includegraphics[width=\linewidth]{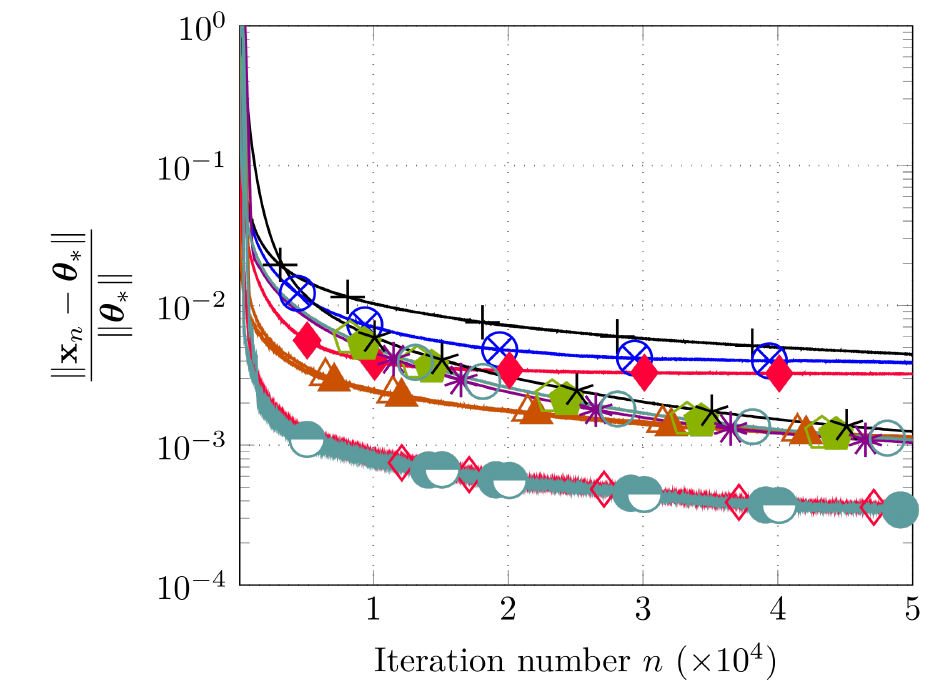}
    \subcaption{Sparsity level: $1\%$.}\label{fig:Sce5}
  \end{subfigure}
  \begin{subfigure}[t]{.45\linewidth}
    \centering
    \includegraphics[width=\linewidth]{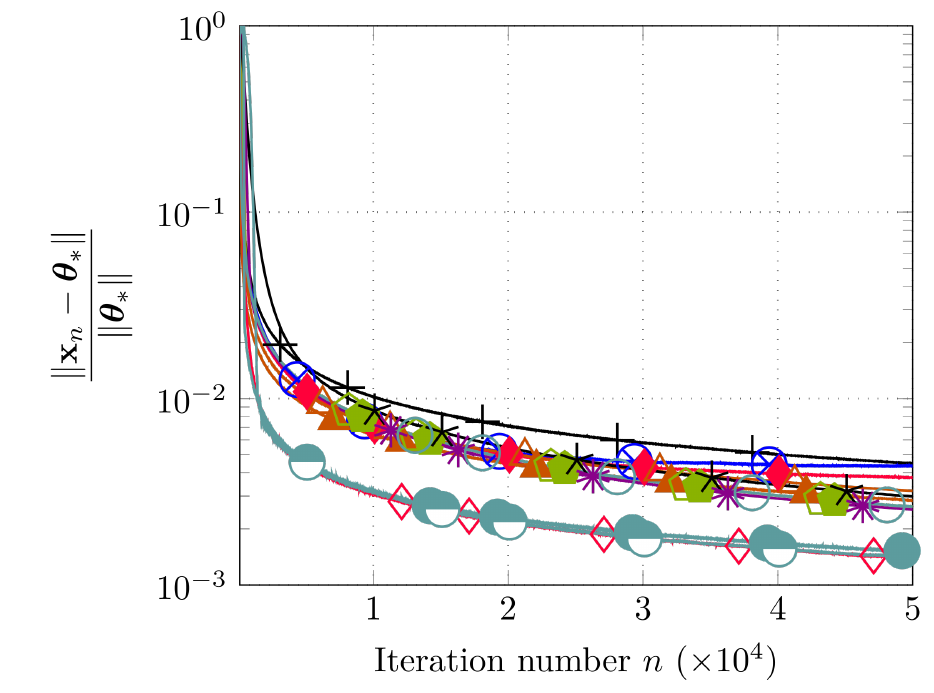}
    \subcaption{Sparsity level: $10\%$.}\label{fig:Sce6}
  \end{subfigure}
  \caption{AR $(\vect{a}_n)_n$; $\text{SNR} = 20\text{dB}$.}\label{fig:Sce5-6}
\end{figure}

\begin{figure}	
  \centering
  \begin{subfigure}[t]{.45\linewidth}
    \centering
    \includegraphics[width=\linewidth]{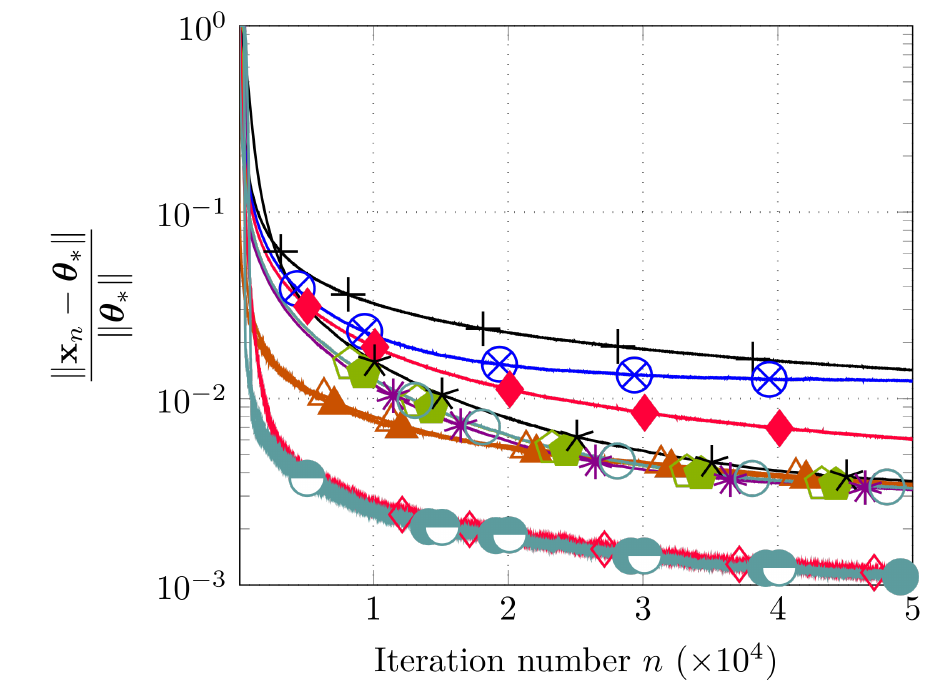}
    \subcaption{Sparsity level: $1\%$.}\label{fig:Sce7}
  \end{subfigure}
  \begin{subfigure}[t]{.45\linewidth}
    \centering
    \includegraphics[width=\linewidth]{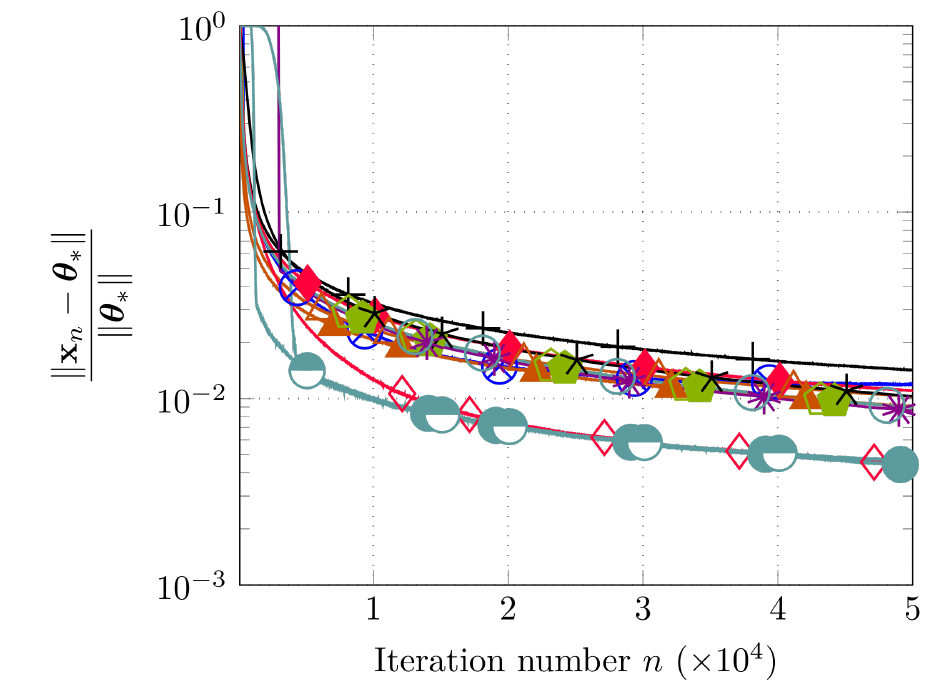}
    \subcaption{Sparsity level: $10\%$.}\label{fig:Sce8}
  \end{subfigure}
  \caption{AR $(\vect{a}_n)_n$; $\text{SNR} = 10\text{dB}$.}\label{fig:Sce7-8}
\end{figure}


\section{Conclusions}\label{Sec:conclusions}

This paper presented a novel stochastic-approximation tool, namely the stochastic Fej\'{e}r-monotone
hybrid steepest descent method (S-FM-HSDM), to solve convex and affinely constrained composite
minimization tasks. Noise contaminates the information about the task, affecting not only the loss
terms but also the affine constraints. S-FM-HSDM provides a novel way of dealing with stochastic
affine constraints via fixed-point sets of appropriate mappings, while retaining several desirable
properties such as splitting of variables and constant step size. A performance analysis is also
provided to identify the conditions under which the sequence of random variables, generated by
S-FM-HSDM, converges a.s.\ to solutions of the latent noiseless minimization task. Several
offsprings of S-FM-HSDM were presented in the context of a well-studied convexly regularized
least-squares task. The versatility of S-FM-HSDM toward affine constraints opens the door to
computationally efficient novel designs, called hierarchical recursive least squares, which,
according to extensive numerical tests on synthetic data, appear to score the lowest estimation
error across a variety of scenarios versus several state-of-the-art adaptive-filtering,
stochastic-approximation and online-learning schemes. Due to space limitations, rates of
convergence, other theoretical contributions and additional applications of S-FM-HSDM will be
presented elsewhere.

\appendices

\section{Mathematical Preliminaries}\label[appendix]{app:preliminaries}

Symbol $\mathcal{X}$ denotes a finite-dimensional Hilbert space, with inner product
$\innerp{\cdot}{\cdot}$ and induced norm $\norm{} \coloneqq \innerp{\cdot}{\cdot}^{1/2}$. Given a
linear operator $U: \mathcal{X}\to \mathcal{X}$, $\range U$ and $\ker U$ denote the range and kernel
spaces of $U$, respectively. Whenever $\mathcal{X} = \Real^D$, the inner product of $\mathcal{X}$ is
the standard dot-vector one:
$\innerp{\vect{x}}{\vect{x}'} \coloneqq \vect{x}^{\intercal} \vect{x}'$,
$\forall (\vect{x}, \vect{x}')\in \mathcal{X}^2$.

Let $(\Omega, \Sigma, \Prob)$ be a probability space, with $\Expect(\cdot)$ denoting
expectation~\cite{Williams.book.91}. Given a sub-$\sigma$-algebra $\mathcal{G}$ of $\Sigma$, the
class of (non-negative) $\mathcal{G}$-measurable functions is denoted by
($(\mathrm{m} \mathcal{G})^+$) $\mathrm{m}\mathcal{G}$~\cite{Williams.book.91}. Given an orthonormal
basis $\Set{b_i}_{i=1}^{\dim\mathcal{X}}$ of $\mathcal{X}$, $x:\Omega\to\mathcal{X}$ is called a
random variable (RV) if there exist RVs $\Set{\gamma^i: \Omega\to\Real}_{i=1}^{\dim\mathcal{X}}$
s.t.\ $x = \sum_{i=1}^{\dim\mathcal{X}} \gamma^i b_i$. To avoid congestion in notations, a lowercase
symbol $x$ denotes both a RV, \ie, $x:\Omega\to \mathcal{X}$ with $x\in\mathrm{m}\Sigma$, and its
realization $x(\omega)$, $\omega\in\Omega$. The abbreviation a.s.\ stands for either "almost
surely," or, "almost sure" with respect to (w.r.t.) $\Omega$, depending on the syntax of the
sentence. A.s.\ convergence of the RV-sequence $(x_n)_n$ to $\bar{x}$ is denoted by
$x_n\limas_n \bar{x}$. For an RV $\gamma: \Omega\to \Real$, let
$\Expect_{\given\mathcal{G}}(\gamma)$ denote the conditional expectation of $\gamma$, conditioned on
$\mathcal{G}$. The conditional expectation $\Expect_{\given\mathcal{G}}(x)$ is defined by
$\Expect_{\given\mathcal{G}}(x) \coloneqq \sum_i \Expect_{\given\mathcal{G}}(\gamma^i)b_i$. Notice
that $\Expect_{\given\mathcal{G}}(x)\in \mathrm{m} \mathcal{G}$~\cite{Williams.book.91}. Moreover,
$\sigma(x)$ denotes the sub-$\sigma$-algebra of $\Sigma$ generated by
$x$~\cite[\S3.8]{Williams.book.91}. For a "random" linear mapping $Q: \mathcal{X}\to \mathcal{X}$
s.t.\ $Qx\in\mathrm{m}\Sigma$, $\forall x\in\mathcal{X}$, let $\Expect_{\given \mathcal{G}}(Q)$
denote the linear mapping
$\Expect_{\given \mathcal{G}}(Q): \mathcal{X} \to\mathcal{X}: x \mapsto \Expect_{\given
  \mathcal{G}}(Q)x \coloneqq \sum_i \gamma^i \Expect_{\given \mathcal{G}}(Qb_i)$. Further, for each
$n$, define the filtration $\mathcal{F}_n \coloneqq \sigma(\Set{x_{\nu}}_{\nu=0}^n)$, \ie, the
sub-$\sigma$-algebra generated by the RVs $\Set{x_{\nu}}_{\nu=0}^n$~\cite[\S10.1]{Williams.book.91}.

Given $\varphi\in \Gamma_0(\mathcal{X})$ [\cf~\cref{as:convex.losses}] and $\epsilon\in\RealPP$, the
$\epsilon$-subdifferential $\partial_{\epsilon} \varphi$ is the set-valued mapping which maps to any
$z\in\mathcal{X}$ all $\epsilon$-subgradients of $\varphi$ at $z$:
$\partial_{\epsilon} \varphi(z) \coloneqq \Set{\xi\in \mathcal{X} \given \varphi(z) +
  \innerp{\xi}{x-z} - \epsilon \leq \varphi(x), \forall x\in\mathcal{X}}$. The graph of
$\partial_{\epsilon}\varphi$ is defined as
$\graph\partial_{\epsilon}\varphi \coloneqq \Set{(z,\xi)\in \mathcal{X}\times \mathcal{X}\given
  \xi\in \partial_{\epsilon}\varphi(z)}$. Symbol $\partial\varphi$ stands for the subdifferential
mapping, defined as
$\partial\varphi(z) \coloneqq \cap_{\epsilon\in\RealPP} \partial_{\epsilon} \varphi(z)$. Moreover,
given $\lambda\in\RealPP$, the proximal mapping $\prox_{\lambda\varphi}: \mathcal{X} \to\mathcal{X}$
is defined as~\cite{HB.PLC.book}
\begin{align}
  z = \prox_{\lambda\varphi}(x)
  & \Leftrightarrow z = \arg\min_{a\in\mathcal{X}} \tfrac{1}{2}\norm{a - x}^2 +
    \lambda\varphi(a) \notag\\
  & \Leftrightarrow \exists \xi\in \partial\varphi(z)\ \text{s.t.}\ z + \lambda \xi =
    x\,. \label{def.proximal}
\end{align}
In the case where $\varphi$ is the indicator function $\iota_{\mathcal{C}}$ for a closed convex set
$\mathcal{C}$, \ie, $\iota_{\mathcal{C}}(x) \coloneqq 0$, if $x\in \mathcal{C}$, and
$\iota_{\mathcal{C}}(x) \coloneqq +\infty$, if $x\notin \mathcal{C}$~\cite{HB.PLC.book}, then
$\prox_{\lambda\varphi}$, for any $\lambda\in\RealPP$, is nothing but the (metric) projection
mapping $P_{\mathcal{C}}$ onto $\mathcal{C}$. The following holds true for any member $T$ of the
family $\mathfrak{T}_{\mathcal{A}}$ in \eqref{T.family}.

\begin{fact}[\protect{\cite[Prop.~2.12]{FM-HSDM.Optim.18}}]\label{fact:U.I-Q}
  The affine constraint $\mathcal{A} = \Fix T = \ker(\Id-Q) + a = \ker U + a$, where
  $a\in \mathcal{A}$ and $U$ stands for the square root of the positive $\Id-Q$, \ie, the (unique)
  positive mapping s.t.\ $U^2 = \Id-Q$~\cite[Thm.~9.4-2]{Kreyszig}. Moreover,
  $\norm{\Id-Q} \leq 1$~\cite[(7)]{FM-HSDM.Optim.18}, and hence, $\norm{U} \leq 1$.
\end{fact}

\begin{fact}[\protect{\cite[Prop.~2.15]{FM-HSDM.Optim.18}}]\label{fact:Ypsilon.star}
  For any $\lambda\in\RealPP$, define
  \begin{align*}
    \mathcal{A}_* 
    & \coloneqq \Set{x\in\mathcal{A} \given \left[\nabla f(x) + \partial (h+g)(x) \right] \cap
      \range U \neq \emptyset}\,,\\  
    \Upsilon_*^{(\lambda)}
    & \coloneqq \Set*{(x,v) \in\mathcal{A} \times \mathcal{X} \given \tfrac{-Uv}{\lambda} \in \nabla
      f(x) + \partial (h+g)(x)}\,.
  \end{align*}
  Then, $x_*$ solves \eqref{the.problem} $\Leftrightarrow x_*\in \mathcal{A}_*$ $\Leftrightarrow
  \exists v_*\in\mathcal{X}$ s.t.\ $(x_*,v_*)\in \Upsilon_*^{(\lambda)}$. 
\end{fact}

The following lemma is used repeatedly in the sequel.

\begin{lemma}\label{lem:expect.inner.prod}
  For any sub-$\sigma$-algebra $\mathcal{G} \subset\Sigma$,
  $\forall (x, x')\in\mathrm{m}\mathcal{G}\times \mathrm{m}\Sigma$,
  $\Expect_{\given \mathcal{G}} \innerp{x}{x'} = \innerp{x}{\Expect_{\given\mathcal{G}}(x')}$
  a.s. Given a linear mapping $Q: \mathcal{X}\to\mathcal{X}$, then,
  $\Expect_{\given \mathcal{G}}(Qx') = Q\Expect_{\given \mathcal{G}}(x')$. Further, if $Q$ is
  ``random,'' in the sense described earlier in this appendix, then
  $\Expect_{\given \mathcal{G}}(Qx) = \Expect_{\given \mathcal{G}}(Q)x$.
\end{lemma}

\begin{proof}
  First, expectations are assumed to exist. Given an orthonormal basis
  $\Set{b_i}_{i=1}^{\dim\mathcal{X}}$ of $\mathcal{X}$, there exist $\Real$-valued RVs
  $\Set{\gamma^i, \gamma'^i}_{i=1}^{\dim\mathcal{X}}$ s.t.\ $x = \sum_i \gamma^ib_i$ and
  $x' = \sum_i \gamma'^ib_i$ a.s. Hence, according to basic properties of conditional
  expectation~\cite[\S9.7(c)(j)]{Williams.book.91},
  $\Expect_{\given\mathcal{G}} \innerp{x}{x'} = \sum_{i,i'} \gamma^i
  \Expect_{\given\mathcal{G}}(\gamma'^{i'}) \innerp{b_i}{b_{i'}} = \innerp{\sum_i
    \gamma^ib_i}{\sum_{i'} \Expect_{\given\mathcal{G}}(\gamma'^{i'}) b_{i'}} =
  \innerp{x}{\Expect_{\given\mathcal{G}}(x')}$ a.s. Further,
  $\Expect_{\given \mathcal{G}}(Qx') = \Expect_{\given \mathcal{G}} (\sum_i \gamma'^i Qb_i) = \sum_i
  \Expect_{\given \mathcal{G}}(\gamma'^i) Qb_i = Q\Expect_{\given \mathcal{G}}(x')$. Similar
  arguments can lead to the final claim of \cref{lem:expect.inner.prod}.
\end{proof}

\section{Proof of \cref{lem:T.LS}}\label[appendix]{app:T.LS}

First, notice that
$\mathcal{A} = \arg\min_{\vect{x}\in \mathcal{X}} \norm{\vect{R}^{1/2} \vect{x} -
  \vect{R}^{\dagger/2} \vect{r}}^2$, where $\vect{R}^{1/2}$ is the square root of $\vect{R}$ and
$\dagger$ stands for the Moore-Penrose pseudoinverse operation. The previous equality can be
established by observing that $\vect{R}^{1/2}\vect{R}^{\dagger/2}$ is the orthogonal projection
mapping onto the range space $\range\vect{R}^{1/2} = \range\vect{R}$, and that
$\vect{r}\in \range\vect{R}$ due to the normal equations. The claim that mappings \eqref{T.LS}
belong to $\mathfrak{T}_{\mathcal{A}}$ follows then directly from~\cite[(70a) and
(70d)]{FM-HSDM.Optim.18}. The claim of \cref{lem:T.LS} with regards to the mappings \eqref{Tn.LS}
can be also established in a similar way; details are omitted.

\section{Proof of \cref{main.thm}}\label[appendix]{app:main.thm}

Line~\ref{SFMHSDM.step.n.half} of \cref{algo:SFMHSDM} yields
\begin{alignat}{2}
  x_{n+3/2} - x_{n+1/2}
  & {} = {} && T_{n+1}x_{n+1} - T_n^{(\alpha)} x_n \notag \\
  &&& - \lambda \nabla f_{n+1} (x_{n+1}) + \lambda\nabla f_n(x_n)\,. \label{diff.halves}
\end{alignat}
By line~\ref{SFMHSDM.prox.step.n} of \cref{algo:SFMHSDM} and \eqref{def.proximal},
$\exists\xi_n\in \partial (h_{n-1}+g)(x_n)$, or, equivalently
\begin{align}
  (x_n, \xi_n)\in \graph\partial(h_{n-1}+g) \,,
  \label{existence.xi}
\end{align}    
s.t.\ $x_{n-1/2} = x_n + \lambda \xi_n$, $\forall n\in\IntegerPP$. Moreover,
lines~\ref{SFMHSDM.step.half} and \ref{SFMHSDM.prox.step.1} of \cref{algo:SFMHSDM}, as well as
\eqref{diff.halves} and \eqref{existence.xi} suggest
\begin{subequations}\label{diff.eqs.x}
  \begin{alignat}{2}
    x_1
    & {} = {} && T_0^{(\alpha)} x_0 - \lambda \left[\nabla f_0(x_0) + \xi_1\right]\,,
    \label{diff.eqs.x.1}\\ 
    x_{n+2} - x_{n+1}
    & = && T_{n+1} x_{n+1} - T_n^{(\alpha)} x_n \notag\\
    &&& - \lambda \left[\nabla f_{n+1}(x_{n+1}) + \xi_{n+2} \right] \notag\\
    &&& + \lambda \left[\nabla f_n(x_{n}) + \xi_{n+1} \right] \,. \label{diff.eqs.x.n+1}
  \end{alignat}
\end{subequations}
By telescoping \eqref{diff.eqs.x},
\begin{alignat*}{2}
  x_{n+1} 
  & {} = {} && T_nx_n - \sum_{\nu=1}^{n-1} (T_{\nu}^{(\alpha)} - T_{\nu})x_{\nu} - \lambda
  \left[\nabla f_n(x_n) + \xi_{n+1}\right] \\ 
  & = && 2T_{n+1}^{(\alpha)} x_{n+1} - T_{n+1} x_{n+1} + (T_{n}^{(\alpha)}x_n -
  T_{n+1}^{(\alpha)}x_{n+1}) \\
  &&& - \sum_{\nu=1}^{n+1} (T_{\nu}^{(\alpha)} -T_{\nu}) x_{\nu} - \lambda \left[\nabla
    f_n(x_n) + \xi_{n+1}\right]\,, 
\end{alignat*}
or, equivalently, via $T_{\nu}^{(\alpha)} - T_{\nu} = (1-\alpha) (\Id-T_{\nu})$,
\begin{align}
  & (\Id + T_{n+1} - 2T_{n+1}^{(\alpha)})x_{n+1} + (T_{n+1}^{(\alpha)} x_{n+1} - T_{n}^{(\alpha)}
    x_n) \notag\\ 
  & = - (1-\alpha) \sum_{\nu=1}^{n+1} (\Id - T_\nu)x_{\nu} - \lambda
    \left[\nabla f_n(x_n) + \xi_{n+1}\right]\,, \label{first.recursive.eq}
\end{align}
where \eqref{first.recursive.eq} holds true $\forall n\in\IntegerP$. Furthermore,
\begin{alignat}{2}
  &&& \hspace{-1em} (1-2\alpha) (T_{n+1}-\Id) x_{n+1} + Q_{n+1}^{(\alpha)} (x_{n+1} - x_n) \notag\\
  &&& + \alpha (T_{n+1} - T_n)x_n \notag\\
  & {} = {} && (1-2\alpha)(T_{n+1}-\Id) x_{n+1} + (T_{n+1}^{(\alpha)}x_{n+1} -
  T_{n+1}^{(\alpha)}x_n) \notag\\
  &&& + (T_{n+1}^{(\alpha)} - T_{n}^{(\alpha)})x_n \notag\\
  & {} = {} && (\Id + T_{n+1} - 2T_{n+1}^{(\alpha)})x_{n+1} + (T_{n+1}^{(\alpha)} x_{n+1} -
  T_{n}^{(\alpha)} x_n) \notag\\
  & \stackrel{\eqref{first.recursive.eq}}{=} && - w_{n+1} - \lambda \left[\nabla f_n(x_n) +
    \xi_{n+1}\right] \,,\label{gradf.and.ksi}
\end{alignat}
where $\forall n\in\IntegerPP$,
\begin{align}
  w_{n+1} \coloneqq (1-\alpha) \sum\nolimits_{\nu=1}^{n+1} (\Id - T_\nu)x_{\nu} \,. \label{define.w}
\end{align}
Moreover, given $x_*\in\mathcal{A}$, define $\forall n\in\IntegerPP$,
\begin{align}
  v_{n+1} \coloneqq (1-\alpha) \sum\nolimits_{\nu=1}^{n+1} U(x_{\nu}-x_*) \,. \label{define.v}
\end{align}
where $U$ is defined in \cref{fact:U.I-Q}. Also, let $v_0 \coloneqq 0 \eqqcolon w_0$. Notice that
$v_{n+1}$ does not depend on the choice of $x_*\in\mathcal{A}$, since $\forall x_*'\in \mathcal{A}$,
with $x_*'\neq x_*$, \cref{fact:U.I-Q} yields $x_*'-x_*\in \ker U$, and
\begin{align*}
  v_{n+1} & = (1-\alpha) \sum\nolimits_{\nu=1}^{n+1} U(x_{\nu}-x_*'+x_*'-x_*) \\
          & = (1-\alpha) \sum\nolimits_{\nu=1}^{n+1} U(x_{\nu}-x_*')\,.
\end{align*}
Notice again by \cref{fact:U.I-Q} that $x_*\in\mathcal{A} \Leftrightarrow (\Id-T)x_* = 0$. Then, by
\eqref{define.w} and \eqref{define.v}, $\forall n\in\IntegerPP$,
\begin{alignat}{2}
  w_{n+1} & {} = {} && (1-\alpha) \sum\nolimits_{\nu=1}^{n+1} (T-T_{\nu})x_{\nu} \notag\\
  &&& + (1-\alpha) \sum\nolimits_{\nu=1}^{n+1} [(\Id - T)x_{\nu}-(\Id - T)x_*] \notag\\
  & {} = {} && (1-\alpha) \sum\nolimits_{\nu=1}^{n+1} (T-T_{\nu})x_{\nu} \notag\\
  &&& + (1-\alpha) \sum\nolimits_{\nu=1}^{n+1} (\Id - Q)(x_{\nu}-x_*) \notag\\
  & = && (1-\alpha) \sum\nolimits_{\nu=1}^{n+1} (T-T_{\nu})x_{\nu} + Uv_{n+1} \,. \label{w.expression}
\end{alignat}

Arbitrarily fix, now, $(x_*, v_*)\in \Upsilon_*^{(\lambda)}$ of \cref{fact:Ypsilon.star}:
$(\Id-T)x_* = 0$ and $\exists \xi_*\in \partial (h + g)(x_*)$ s.t.\
$Uv_* + \lambda[\nabla f (x_*) + \xi_*] = 0$. Then, by \eqref{gradf.and.ksi}, \eqref{w.expression},
\begin{alignat}{2}
  && (1-2\alpha) & (T_{n+1} - T)x_{n+1} + (1-2\alpha)(T-\Id)x_{n+1} \notag\\
  &&& + \alpha(Q_{n+1} - Q)(x_{n+1} - x_n) \notag\\ 
  &&& + Q^{(\alpha)}(x_{n+1} - x_n) + \alpha (T_{n+1} - T_n)x_n \notag\\
  && {} = {} & - (1-\alpha) \sum\nolimits_{\nu=1}^{n+1} (T - T_\nu)x_{\nu} \notag\\
  &&& - U(v_{n+1}-v_*) -\lambda\left[\nabla f_n(x_n) - \nabla f(x_*) \right] \notag\\
  &&& - \lambda(\xi_{n+1} - \xi_*) \notag\\
  \Leftrightarrow {} && (1-2\alpha) & (\Id-T)x_{n+1} + Q^{(\alpha)}(x_n-x_{n+1}) \notag\\
  &&& + U(v_*-v_{n+1}) + (1-2\alpha) (T-T_{n+1}) x_{n+1} \notag\\
  &&& + \alpha(Q-Q_{n+1})(x_{n+1} - x_n) \notag\\
  &&& + \alpha (T_n-T_{n+1})x_n \notag\\
  &&& + (1-\alpha) \sum\nolimits_{\nu=1}^{n+1} (T_{\nu}-T)x_{\nu}  \notag\\
  && {} = {} & \lambda \left[ \nabla f_n(x_n) - \nabla f(x_*) + \xi_{n+1} - \xi_* \right]
  \,. \label{to.use.in.BH}
\end{alignat}
Based on \cref{as:f.Lipschitz}, the application of the Baillon-Haddad theorem~\cite[Cor.~18.16]{HB.PLC.book} to $f$ suggests that%
\begin{subequations}%
  \begin{alignat}{2}
    &&& \hspace{-2em} \tfrac{2\lambda}{L_{\nabla f}} \norm{\nabla f(x_n) - \nabla f(x_*)}^2 \notag \\
    & {} \leq {} && 2\lambda \innerp{x_n-x_*}{\nabla f(x_n) - \nabla
      f(x_*)} \notag \\
    & = &&
    2\lambda \innerp{x_{n+1}-x_*}{\nabla f_n(x_n) - \nabla f(x_*) + \xi_{n+1}-\xi_*} \notag\\
    &&& + 2\lambda \innerp{x_{n+1}-x_*}{(\nabla f-\nabla f_n)x_n} \notag\\
    &&& + 2\lambda \innerp{x_n-x_{n+1}}{\nabla f(x_n) - \nabla f(x_*)} \notag\\
    &&& + 2\lambda \innerp{x_*-x_{n+1}}{\xi_{n+1}-\xi_*} \notag\\
    & \stackrel{\eqref{to.use.in.BH}}{=} && 2 (1-2\alpha) \innerp{x_{n+1}-x_*}{(\Id-T)x_{n+1}}
    \notag\\
    &&& + 2\innerp{x_{n+1}-x_*}{Q^{(\alpha)}(x_n-x_{n+1})} \notag\\
    &&& + 2\innerp{x_{n+1}-x_*}{U(v_*-v_{n+1})} \notag\\
    &&& + 2\lambda \innerp{x_n-x_{n+1}}{\nabla f(x_n) - \nabla f(x_*)} \notag\\
    &&& + 2 (1-2\alpha) \innerp{x_{n+1}-x_*}{(T-T_{n+1})x_{n+1}} \notag\\
    &&& + 2\alpha \innerp{x_{n+1}-x_*}{(Q-Q_{n+1})(x_{n+1} - x_n)} \notag\\
    &&& + 2\alpha \innerp{x_{n+1}-x_*}{(T_n-T_{n+1})x_n} \notag\\
    &&& + 2(1-\alpha) \innerp{x_{n+1}-x_*}{\sum\nolimits_{\nu=1}^{n+1}(T_{\nu}-T)x_{\nu}} \notag\\
    &&& + 2\lambda \innerp{x_{n+1}-x_*}{(\nabla f-\nabla f_n)x_n} \notag\\
    &&& + 2\lambda \innerp{x_*-x_{n+1}}{\xi_{n+1}-\xi_*} \notag\\
    & \leq && 2 (1-2\alpha) \innerp{x_{n+1}-x_*}{(\Id-T)x_{n+1} - (\Id-T)x_{*}} \notag\\
    &&& + 2\innerp{x_{n+1}-x_*}{Q^{(\alpha)}(x_n-x_{n+1})} \notag\\
    &&& + 2\innerp{x_*-x_{n+1}}{U(v_{n+1}-v_*)} \notag\\
    &&& + \tfrac{\lambda L_{\nabla f}}{2} \norm{x_n-x_{n+1}}^2 + \tfrac{2\lambda}{L_{\nabla f}}
    \norm{\nabla f(x_n) - \nabla f(x_*)}^2 \label{Young.inequality}\\
    &&& + 2 (1-2\alpha) \innerp{x_{n+1}-x_*}{(T-T_{n+1})x_{n+1}}
    \label{vartheta.start}\\
    &&& + 2\alpha \innerp{x_{n+1}-x_*}{(Q-Q_{n+1})(x_{n+1} - x_n)} \notag\\
    &&& + 2\alpha \innerp{x_{n+1}-x_*}{(T_n-T_{n+1})x_n} \notag\\
    &&& + 2(1-\alpha) \innerp{x_{n+1}-x_*}{\sum\nolimits_{\nu=1}^{n+1} (T_{\nu}-T)x_{\nu}} \notag\\
    &&& + 2\lambda \innerp{x_{n+1}-x_*}{(\nabla f-\nabla f_n)x_n} \notag \\
    &&& + 2\lambda \innerp{x_*-x_{n+1}}{\xi_{n+1}-\xi_*}\,, \label{BH.thm}
  \end{alignat}%
\end{subequations}%
where $2\innerp{\sqrt{\beta}a}{b/\sqrt{\beta}} \leq \beta\norm{a}^2 + \norm{b}^2/\beta$,
$a\coloneqq x_n-x_{n+1}$, $b\coloneqq \nabla f(x_n) - \nabla f(x_*)$ and
$\beta \coloneqq L_{\nabla f}/2$, were used in \eqref{Young.inequality}. Let $\vartheta_n$ be the RV
defined by the expression which starts from \eqref{vartheta.start} and ends at \eqref{BH.thm}.

Let $y \coloneqq (x,v)$ denote an element of the finite-dimensional Hilbert space
$(\mathcal{X}^2, \innerp{\cdot}{\cdot}_{\mathcal{X}^2})$, where the inner product is defined as
$\innerp{(x, v)}{(x', v')}_{\mathcal{X}^2} \coloneqq \innerp{x}{x'} + \innerp{v}{v'}$, for any
$(x,v), (x',v')\in\mathcal{X}^2$. Let also the bounded linear and self-adjoint operator
$\Theta: \mathcal{X}^2 \to \mathcal{X}^2: (x,v) \mapsto (Q^{(\alpha)}x, v/(1-\alpha))$. By virtue of
the positivity of $Q$, $\forall x$,
\begin{align}
  \innerp{Q^{(\alpha)}x}{x}
  & = \alpha \innerp{Qx}{x} + (1-\alpha)\norm{x}^2 \notag\\
  & \geq (1-\alpha)\norm{x}^2\,, \label{Qalpha.strongly.pos}
\end{align}
which renders $\Theta$ strongly positive (recall $\alpha<1$). Operator $\Theta$ induces thus the
Hilbert space
$\mathcal{X}^2_{\Theta} \coloneqq (\mathcal{X}^2, \innerp{\cdot}{\cdot}_{\mathcal{X}^2_{\Theta}})$
with inner product
$\innerp{\cdot}{\cdot}_{\mathcal{X}^2_{\Theta}} \coloneqq
\innerp{\Theta(\cdot)}{\cdot}_{\mathcal{X}^2}$. Then, upon defining $y_* \coloneqq (x_*, v_*)$,
\eqref{BH.thm} yields%
\begin{subequations}%
  \begin{alignat}{2}
    0 & {} \leq {} && 2 (1-2\alpha) \innerp{x_{n+1}-x_*}{(\Id-Q)(x_{n+1}-x_*)} \notag \\
    &&& + 2\innerp{Q^{(\alpha)}(x_n-x_{n+1})}{x_{n+1}-x_*} \notag \\
    &&& + 2\innerp{U(x_*-x_{n+1})}{v_{n+1}-v_*} \notag\\
    &&& + \tfrac{\lambda L_{\nabla f}}{2}\norm{x_n-x_{n+1}}^2 + \vartheta_n \notag\\
    & \leq && 2\innerp{Q^{(\alpha)}(x_n-x_{n+1})}{x_{n+1}-x_*} + \tfrac{\lambda L_{\nabla f}}{2}
    \norm{x_n-x_{n+1}}^2 \notag\\
    &&&  + \tfrac{2}{1-\alpha} \innerp{v_n-v_{n+1}}{v_{n+1}-v_*} + \vartheta_n \label{vartheta.1}\\
    & = && 2\innerp{\Theta(y_n-y_{n+1})}{y_{n+1}-y_*}_{\mathcal{X}^2}  \notag\\
    &&& + \tfrac{\lambda L_{\nabla f}}{2} \norm{x_n-x_{n+1}}^2 + \vartheta_n \notag\\  
    & = && \norm{y_n-y_*}_{\mathcal{X}^2_{\Theta}}^2 - \norm{y_{n+1}-y_*}_{\mathcal{X}^2_{\Theta}}^2 -
    \norm{y_{n+1}-y_n}_{\mathcal{X}^2_{\Theta}}^2 \notag\\
    &&& + \tfrac{\lambda L_{\nabla f}}{2} \norm{x_n-x_{n+1}}^2 + \vartheta_n \notag\\
    & \leq  && \norm{y_n-y_*}_{\mathcal{X}^2_{\Theta}}^2 -
    \norm{y_{n+1}-y_*}_{\mathcal{X}^2_{\Theta}}^2 \notag\\ 
    &&& - (1-\zeta) \norm{y_{n+1}-y_n}_{\mathcal{X}^2_{\Theta}}^2 +
    \vartheta_n \label{vartheta.2}\,,
  \end{alignat}%
\end{subequations}%
where the positivity of $\Id-Q$ from \cref{fact:U.I-Q}, $\alpha\geq 1/2$, and
$v_n-v_{n+1} = (1-\alpha) U(x_*-x_{n+1})$, from \eqref{define.v}, were used in
\eqref{vartheta.1}. Any $\zeta\in (\lambda L_{\nabla f}/[2(1-\alpha)],1)$ justifies
\eqref{vartheta.2}, since for $\lambda < 2(1-\alpha)/L_{\nabla f}$ and
$\forall y \coloneqq (x, v)\in \mathcal{X}^2$, \eqref{Qalpha.strongly.pos} suggests
\begin{align*}
  \tfrac{\lambda L_{\nabla f}}{2} \norm{x}^2
  & < \zeta(1-\alpha) \norm{x}^2 \leq \zeta \innerp{x}{Q^{(\alpha)}x} \\
  & \leq \zeta \innerp{x}{Q^{(\alpha)}x} + \zeta\tfrac{1}{1-\alpha}\norm{v}^2 = \zeta
    \norm{y}^2_{\mathcal{X}^2_{\Theta}} \,.
\end{align*}
Notice by \eqref{define.v} that $v_n \in\mathrm{m}\mathcal{F}_n$. Hence,
$y_n = (x_n, v_n) \in\mathrm{m}\mathcal{F}_n$. Applying $\Expect_{\given \mathcal{F}_n}(\cdot)$ to
\eqref{vartheta.2} yields%
\begin{alignat}{2}
  \Expect_{\given \mathcal{F}_n} &&& \norm{y_{n+1}-y_*}_{\mathcal{X}_{\Theta}^2}^2 + (1-\zeta)
  \Expect_{\given \mathcal{F}_n} \norm{y_{n+1}-y_n}_{\mathcal{X}_{\Theta}^2}^2 \notag\\
  &&& \leq \norm{y_n-y_*}_{\mathcal{X}_{\Theta}^2}^2 + \Expect_{\given \mathcal{F}_n}
  (\vartheta_n)^+ \quad\text{a.s.,} \label{Fejer.cond} 
\end{alignat}
where
$\Expect_{\given \mathcal{F}_n} (\vartheta_n)^+ \coloneqq \max\Set{0, \Expect_{\given \mathcal{F}_n}
  (\vartheta_n)}$. Since $(x_*, v_*)$ was arbitrarily chosen from $\Upsilon_*^{(\lambda)}$ (\cf
\cref{fact:Ypsilon.star}), \cref{as:domination} and \cite[Prop.~2.3]{PLC.JCP.stochastic.QF.15}
render $(y_n)_n$ stochastic quasi-Fej\'{e}r monotonous w.r.t.\ $\Upsilon_*^{(\lambda)}$; thus,
bounded a.s. Due to $y_n = (x_n, v_n)$, sequences $(x_n)_n$ and $(v_n)_n$ are also bounded a.s.

This paragraph proves the claim that \eqref{strict.conditions} suffice for \cref{as:domination} to
hold true. Via \eqref{unbiased.subgrad.hg}, a.s.,
\begin{align*}
  & (h+g)(\vect{x}_*) \\
  & \geq (h+g)(\vect{x}_{n+1}) + \innerp{\vect{x}_*- \vect{x}_{n+1}}{\Expect_{\given
    \mathcal{F}_{n+1}}(\bm{\xi}_{n+1})} - \epsilon_{n+1} \,.
\end{align*}
Moreover,
$(\vect{x}_*, \bm{\xi}_*)\in \graph\partial(h+g) \Rightarrow (h+g)(\vect{x}_{n+1}) \geq
(h+g)(\vect{x}_*) + \innerp{\vect{x}_{n+1}- \vect{x}_*}{\bm{\xi}_*}$. Hence, by adding the previous
two inequalities,
$\innerp{\vect{x}_*- \vect{x}_{n+1}}{\Expect_{\given \mathcal{F}_{n+1}}(\bm{\xi}_{n+1})- \bm{\xi}_*}
\leq \epsilon_{n+1}$. The ``tower property'' of conditional probability suggests
$\Expect_{\given \mathcal{F}_n} (\vartheta_n) = \Expect_{\given \mathcal{F}_n} \Expect_{\given
  \mathcal{F}_{n+1}}(\vartheta_n)$~\cite[\S9.7(i)]{Williams.book.91}. By
$\vect{x}_{n+1}- \vect{x}_*\in \mathrm{m}\mathcal{F}_{n+1}$, \eqref{strict.conditions},
\eqref{BH.thm} and \cref{lem:expect.inner.prod},
\begin{alignat}{2}
  &&& \hspace{-1.25em} \Expect_{\given \mathcal{F}_{n+1}} (\vartheta_n) \notag\\
  && {} = {} & 2 (1-2\alpha) \innerp{\vect{x}_{n+1} - \vect{x}_*}{\Expect_{\given
      \mathcal{F}_{n+1}}[(T-T_{n+1}) \vect{x}_{n+1}]} \notag\\
  &&& + 2\alpha \innerp{\vect{x}_{n+1} - \vect{x}_*}{\Expect_{\given
      \mathcal{F}_{n+1}}[(Q-Q_{n+1})(\vect{x}_{n+1} - \vect{x}_n)]} \notag\\
  &&& + 2\alpha \innerp{\vect{x}_{n+1} - \vect{x}_*}{\Expect_{\given
      \mathcal{F}_{n+1}}[(T_n-T_{n+1}) \vect{x}_n]} \notag\\  
  &&& + 2(1-\alpha) \innerp{\vect{x}_{n+1} - \vect{x}_*}{-\vect{t}_{n+1}} \notag\\
  &&& + 2\lambda \innerp{\vect{x}_{n+1} - \vect{x}_*}{\Expect_{\given \mathcal{F}_{n+1}} [(\nabla f - \nabla
    f_{n+1}) \vect{x}_n]} \notag \\ 
  &&& + 2\lambda \innerp{\vect{x}_{n+1} - \vect{x}_*}{\Expect_{\given \mathcal{F}_{n+1}} [(\nabla
    f_{n+1} - \nabla f_n) \vect{x}_n]} \notag \\
  &&& + 2\lambda \innerp{\vect{x}_* - \vect{x}_{n+1}}{\Expect_{\given
      \mathcal{F}_{n+1}}(\bm{\xi}_{n+1})- \bm{\xi}_*} \leq 2\lambda \epsilon_{n+1}
  \,. \label{prove.suff.cond.theta} 
\end{alignat}
Thus,
$\Expect_{\given \mathcal{F}_{n}} (\vartheta_n)\leq 2\lambda \Expect_{\given
  \mathcal{F}_{n}}(\epsilon_{n+1}) \Rightarrow \Expect_{\given \mathcal{F}_n} (\vartheta_n)^+ \leq
2\lambda \Expect_{\given \mathcal{F}_{n}}(\epsilon_{n+1}) \Rightarrow \sum_n \Expect_{\given
  \mathcal{F}_n} (\vartheta_n)^+ \leq 2\lambda \sum_n \Expect_{\given \mathcal{F}_{n}}(\epsilon_{n+1})$
a.s. By \eqref{unbiased.subgrad.hg},
$\sum_n \Expect[\Expect_{\given \mathcal{F}_{n}}(\epsilon_{n+1})] = \sum_n \Expect(\epsilon_{n+1}) <
+\infty$, and $\psi \coloneqq 2\lambda \sum_n \Expect_{\given \mathcal{F}_{n}}(\epsilon_{n+1})$, a.s.,
satisfies \cref{as:domination}.

Going back to the general setting, define now space
$\mathfrak{X} \coloneqq L^2[(\Omega, \Sigma, \Prob), \mathcal{X}]$ of (equivalent classes of Borel)
measurable functions, or, RVs $x: \Omega\to \mathcal{X}$ s.t.\
$\int_{\Omega} \norm{x(\omega)}^2 \Prob(d\omega) < +\infty$.  This RV-space $\mathfrak{X}$ turns out
to be a real Hilbert one with inner product
$\innerp{x}{x'}_{\mathfrak{X}} \coloneqq \Expect(\innerp{x}{x'}) \coloneqq \int_{\Omega}
\innerp{x(\omega)}{x'(\omega)} \Prob(d\omega)$,
$\forall (x,x')\in \mathfrak{X} \times \mathfrak{X}$~\cite[Ex.~2.5, p.~28]{HB.PLC.book}. Hilbert
space $\mathfrak{X}^2_{\Theta} \coloneqq L^2[(\Omega, \Sigma, \Prob), \mathcal{X}^2_{\Theta}]$ is
similarly defined, with inner product
$\innerp{y}{y'}_{\mathfrak{X}} \coloneqq \Expect(\innerp{y}{y'}_{\mathcal{X}^2_{\Theta}})$.

Application of $\Expect(\cdot)$ to \eqref{Fejer.cond}, under the light of
$\Expect(\cdot) = \Expect[\Expect_{\given \mathcal{F}_n}(\cdot)]$~\cite[\S9.7(a)]{Williams.book.91},
yields
\begin{align}
  \norm{y_{n+1}-y_*}_{\mathfrak{X}_{\Theta}^2}^2 & + (1-\zeta) 
  \norm{y_{n+1}-y_n}_{\mathfrak{X}_{\Theta}^2}^2 \notag\\
  & \leq \norm{y_n-y_*}_{\mathfrak{X}_{\Theta}^2}^2 + \Expect[\Expect_{\given \mathcal{F}_n}
    (\vartheta_n)^+]  \,. \label{Fejer.cond.type3} 
\end{align}
The monotone-convergence theorem~\cite[\S5.3]{Williams.book.91} and \cref{as:domination} imply that
$\sum_{n} \Expect [\Expect_{\given\mathcal{F}_n} (\vartheta_n)^+] \leq \Expect(\psi) < +\infty$. As
such, \eqref{Fejer.cond.type3} renders $(y_n)_n$ quasi-Fej\'{e}r (of type~III) w.r.t.\
$\Upsilon_*^{(\lambda)}$ and, thus, bounded within
$\mathfrak{X}^2_{\Theta}$~\cite[Prop.~3.3]{PLC.quasi.Fejer.2001}. Hence, both $(x_n)_n$ and
$(v_n)_n$ are bounded within $\mathfrak{X}$. Moreover, by telescoping \eqref{Fejer.cond.type3},
$\forall n$,
$\sum_{\nu=0}^{n} \norm{y_{\nu+1}-y_{\nu}}_{\mathfrak{X}_{\Theta}^2}^2 \leq (1-\zeta)^{-1}
[\norm{y_0-y_*}_{\mathfrak{X}_{\Theta}^2}^2 + \Expect(\psi)]$. Hence,
$\sum_{n=0}^{+\infty} \Expect (\norm{y_{n+1}-y_n}_{\mathcal{X}_{\Theta}^2}^2) < +\infty$. By
\cite[\S6.5]{Williams.book.91}, $\norm{y_{n+1}-y_n}_{\mathcal{X}_{\Theta}^2}^2 \limas_n 0$, and
thus, $(y_{n+1}-y_n) \limas_n 0$ by virtue of the strong positivity of $\Theta$. Consequently,
\begin{align}
  (x_{n+1}-x_n) \limas_n 0\,, \quad (v_{n+1}-v_n) \limas_n 0\,; \label{x.v.as.regular}
\end{align}
hence, both $(x_{n+1}-x_n)_n$ and $(v_{n+1}-v_n)_n$ are bounded a.s.

By \eqref{w.expression},
\begin{alignat*}{2}
  &&& \hspace{-2em} w_{n+1} - w_{n} \\
  & {} = {} && (1-\alpha) (T-T_{n+1})x_{n+1} + U(v_{n+1}- v_n) \\
  & = && (1-\alpha) (Tx_{n+1} - Tx_n) + (1-\alpha) (T-T_{n+1})x_n  \\
  &&& + (1-\alpha) (T_{n+1}x_n - T_{n+1}x_{n+1}) + U(v_{n+1}- v_n) \\
  & = && (1-\alpha) Q(x_{n+1}-x_n) + (1-\alpha) (Q-Q_{n+1})x_n \\
  &&& + (1-\alpha) (\pi-\pi_{n+1}) + (1-\alpha) Q_{n+1}(x_n-x_{n+1}) \\
  &&& + U(v_{n+1}- v_n) \,.
\end{alignat*}
Since $(x_n)_n$ is bounded a.s., there exists $C_1 \coloneqq C_1(\omega)\in \RealPP$ s.t.\ $\norm{x_n}\leq
C_1$, $\forall n$, a.s. Consequently, by \cref{fact:U.I-Q}, \cref{as:T,as:Tn},
\begin{alignat}{2}
  &&& \hspace{-2em} \norm{w_{n+1} - w_{n}} \notag\\
  & {} \leq {} && (1-\alpha) \norm{Q}\,\norm{x_{n+1}-x_n} + (1-\alpha) \norm{Q-Q_{n+1}}\,
  \norm{x_n} \notag\\ 
  &&& + (1-\alpha) \norm{\pi-\pi_{n+1}} + (1-\alpha) \norm{Q_{n+1}}\, \norm{x_n-x_{n+1}} \notag\\
  &&& + \norm{U}\, \norm{v_{n+1}- v_n} \notag\\
  & \leq && \norm{x_{n+1}-x_n} + C_1 \norm{Q-Q_{n+1}} + \norm{\pi-\pi_{n+1}} \notag\\
  &&& + \norm{x_n-x_{n+1}} + \norm{v_{n+1}- v_n} \,, \label{consecutive.w}
\end{alignat}
Via \cref{as:consistent.Tn} and \eqref{x.v.as.regular}, \eqref{consecutive.w} yields
\begin{align}
  w_{n+1} - w_{n}\limas_n 0\,. \label{wn.as.regular}
\end{align}
Hence, for any $\epsilon\in\RealPP$, there exists
$n_{\mypound} \coloneqq n_{\mypound}(\omega) \in\IntegerP$ s.t.\ $\forall n\geq n_{\mypound}$,
$\norm{w_{n+1} - w_{n}}\leq \epsilon$ a.s. Notice also via Jensen's
inequality~\cite[\S9.7(h)]{Williams.book.91} that
$\norm{\Expect_{\given \mathcal{F}_n} (w_{n+1}) - \Expect_{\given \mathcal{F}_n} (w_{n})} =
\norm{\Expect_{\given \mathcal{F}_n} (w_{n+1} - w_{n})} \leq \Expect_{\given \mathcal{F}_n}
(\norm{w_{n+1} - w_{n}}) \leq \epsilon$, and thus,
$\limsup_n \norm{\Expect_{\given \mathcal{F}_n} (w_{n+1}) - \Expect_{\given \mathcal{F}_n} (w_{n})}
\leq \epsilon$ a.s. Since $\epsilon$ is chosen arbitrarily,
\begin{align}
  \Expect_{\given \mathcal{F}_n} (w_{n+1}) - \Expect_{\given
    \mathcal{F}_n} (w_{n}) \limas_n0 \,. \label{expect.wn.as.regular}
\end{align}

Furthermore, by \eqref{define.w}, $(\Id-T_{n+1})x_{n+1} = (w_{n+1}-w_n)/(1-\alpha)$, which, together
with \eqref{wn.as.regular}, yields
\begin{align}
  (\Id-T_n)x_n \limas_n 0\,. \label{I.minus.Tn.to.zero}
\end{align}
As such, $((\Id-T_n)x_n)_n$ is bounded a.s. Moreover,
\begin{alignat*}{2}
  &&& \hspace{-2em} \norm{(\Id-T)x_n} \\
  & {} \leq {} && \norm{[(\Id-T) - (\Id-T_n)]x_n} + \norm{(\Id-T_n)x_n} \\
  & \leq && \norm{Q_n-Q}\, \norm{x_n} + \norm{\pi_n-\pi} + \norm{(\Id-T_n)x_n}\\
  & \leq && C_1 \norm{Q_n-Q} + \norm{\pi_n-\pi} + \norm{(\Id-T_n)x_n}\,.
\end{alignat*}
Referring again to \cref{as:consistent.Tn}, \eqref{I.minus.Tn.to.zero} and the previous inequality
yield
\begin{align}
  (\Id-T)x_n \limas_n 0\,. \label{I.minus.T.to.zero}
\end{align}
Moreover,
\begin{alignat*}{2}
  \norm{(T_n - T_{n+1}) x_n} & {} \leq {} && \norm{(Q_n - Q_{n+1}) x_n} + \norm{\pi_n-\pi_{n+1}} \\
  & \leq && \norm{Q_n - Q_{n+1}} \norm{x_n} + \norm{\pi_n-\pi_{n+1}} \\ 
  & \leq && C_1 \norm{Q_n - Q_{n+1}} + \norm{\pi_n-\pi_{n+1}} \,,
\end{alignat*}
which, according to \cref{as:consistent.Tn}, leads to
\begin{align}
  (T_n - T_{n+1}) x_n \limas_n 0\,. \label{Tn.as.regular}
\end{align}
Hence, $((T_n - T_{n+1}) x_n)_n$ is bounded a.s.

\cref{as:fn.Lipschitz,as:bounded.Lips} suggest that for any $z\in\mathcal{X}$,
\begin{align}
  & \norm{\nabla f_{n}(x_{n}) - \nabla f(z)}^2 \notag\\
  & \leq  2\norm{\nabla f_{n}(x_{n}) - \nabla f_{n}(z)}^2 +
    2\norm{(\nabla f_{n} -  \nabla f)z}^2 \notag\\
  & \leq 2L_{n}^2 \norm{x_{n}-z}^2 + 2\norm{(\nabla f_{n}-\nabla f)z}^2 \notag\\
  & \leq 2C_{\text{Lip}}^2 \norm{x_{n}-z}^2 + 2\norm{(\nabla f_{n}-\nabla f)z}^2
    \,.  \label{gradientf.converges}
\end{align}
The a.s.\ boundedness of $(x_n)_n$ implies the a.s.\ boundedness of $(x_n-z)_n$. Moreover,
\cref{as:consistent.gradfn} suggests the a.s.\ boundedness of $((\nabla f_{n}-\nabla f)z)_n$. Due
also to
$\norm{\nabla f_{n}(x_{n})} \leq \norm{\nabla f_{n}(x_{n}) - \nabla f(z)} + \norm{\nabla f(z)}$,
\eqref{gradientf.converges} guarantees that $(\nabla f_{n}(x_{n}))_n$ is bounded a.s. Notice also by
\eqref{gradf.and.ksi},
  \begin{alignat}{2}
    \xi_{n+1} + \tfrac{1}{\lambda} w_{n+1} 
    & {} = {} && \tfrac{1-2\alpha}{\lambda} (\Id - T_{n+1})x_{n+1} \notag\\
    &&& + \tfrac{1}{\lambda} Q_{n+1}^{(\alpha)} (x_n - x_{n+1}) \notag\\
    &&& + \tfrac{\alpha}{\lambda} (T_n - T_{n+1}) x_n - \nabla f_n(x_n)
    \,. \label{xi.n+1}
  \end{alignat}
  Due also to the a.s.\ boundedness of $((\Id - T_{n})x_{n})_n$, $(x_{n+1} - x_{n})_n$,
  $((T_{n+1} - T_{n}) x_n)_n$ and $(\nabla f_n(x_n))_n$, there exists
  $C_2 \coloneqq C_2(\omega)\in\RealPP$ s.t.\
\begin{alignat}{2}
  &&& \hspace{-2em} \norm{\xi_{n+1} + \tfrac{1}{\lambda} w_{n+1}} \notag\\
  & {} \leq {} && \tfrac{2\alpha-1}{\lambda} \norm{(\Id -
    T_{n+1})x_{n+1}} + \tfrac{1}{\lambda} \norm{Q_{n+1}^{(\alpha)} (x_n - x_{n+1})} \notag\\
  &&& + \tfrac{\alpha}{\lambda} \norm{(T_n - T_{n+1}) x_n} + \norm{\nabla f_n(x_n)} \notag\\
  & \leq && \tfrac{2\alpha-1}{\lambda} \norm{(\Id -
    T_{n+1})x_{n+1}} + \tfrac{1}{\lambda} \norm{x_n - x_{n+1}} \notag\\
  &&& + \tfrac{\alpha}{\lambda} \norm{(T_n - T_{n+1}) x_n} + \norm{\nabla f_n(x_n)} 
  \leq C_2 \quad\text{a.s.} \label{bounded.xi.w.as}
\end{alignat}

\begin{lemma}\label{lem:cluster.points}
  The cluster-point set $\mathfrak{C}[(y_n)_n]$ of sequence $(y_n)_n$, as well as
  $\mathfrak{C}[(x_n)_n]$ and $\mathfrak{C}[(v_n)_n]$ are nonempty. If
  $\bar{y} \eqqcolon (\bar{x}, \bar{v}) \in \mathfrak{C}[(y_n)_n]$, then,
  $\bar{x} \in \mathfrak{C}[(x_n)_n]$ and $\bar{v}\in \mathfrak{C}[(v_n)_n]$. For any
  $\bar{x}\in \mathfrak{C}[(x_n)_n]$, there exists $\bar{v}\in \mathfrak{C}[(v_n)_n]$ s.t.\
  $\bar{y} \coloneqq (\bar{x}, \bar{v}) \in \mathfrak{C}[(y_n)_n]$. All of the previous statements
  hold true a.s.
\end{lemma}

\begin{proof}
  Since $(y_n)_n$ is bounded a.s.\ [\cf discussion after \eqref{Fejer.cond}], its set of cluster
  points is nonempty~\cite[Fact~2.26(iii) and Lem.~2.37]{HB.PLC.book}. Moreover, due to the
  boundedness of $(x_n)_n$ and $(v_n)_n$, $\mathfrak{C}[(x_n)_n]$ and $\mathfrak{C}[(v_n)_n]$ are
  also nonempty. For any cluster point $\bar{y} \eqqcolon (\bar{x}, \bar{v}) \in \mathfrak{C}[(y_n)_n]$,
  there exists a subsequence $(n_k)_k$ s.t.\
  $y_{n_k} \coloneqq (x_{n_k}, v_{n_k}) \limas_k (\bar{x}, \bar{v})$, \ie,
  $\bar{x}\in \mathfrak{C}[(x_n)_n]$ and $\bar{v}\in \mathfrak{C}[(v_n)_n]$. On the other hand,
  given any $\bar{x}\in \mathfrak{C}[(x_n)_n]$, there exists a subsequence $(x_{n_k})_k$ s.t.\
  $x_{n_k}\limas_k \bar{x}$. Since $(v_n)_n$ is bounded, passing to a subsequence of $(n_k)_k$ if
  necessary (avoided here to avoid notational congestion), there exists
  $\bar{v}\in \mathfrak{C}[(v_n)_n]$ s.t.\ $v_{n_k}\limas_k \bar{v}$, and thus,
  $y_{n_k} \coloneqq (x_{n_k}, v_{n_k}) \limas_k (\bar{x}, \bar{v}) \eqqcolon \bar{y} \in
  \mathfrak{C}[(y_n)_n]$.
\end{proof}

Choose, now, arbitrarily a cluster point
$\bar{y} \eqqcolon (\bar{x}, \bar{v})\in \mathfrak{C}[(y_n)_n] \neq\emptyset$. Hence, there exists a
subsequence $(n_k)_k$ s.t.\ $y_{n_k} \eqqcolon (x_{n_k}, v_{n_k}) \limas_k (\bar{x},
\bar{v})$. Then, by \eqref{I.minus.T.to.zero}, applied to $(x_{n_k})_k$, and by the nonexpansivity
(thus continuity) of $T$,
\begin{align}
  \bar{x}\in\Fix T = \mathcal{A} \quad \text{a.s.} \label{barx.fixed.point}
\end{align}
Setting $n=n_k$ and $z = \bar{x}$ in \eqref{gradientf.converges} yields
$\norm{\nabla f_{n_k}(x_{n_k}) - \nabla f(\bar{x})}^2 \leq 2C_{\text{Lip}}^2
\norm{x_{n_k}-\bar{x}}^2 + 2\norm{(\nabla f_{n_k}-\nabla f)\bar{x}}^2$, which, by
\cref{as:consistent.gradfn} and $x_{n_k} \limas_k \bar{x}$, deduces
$\nabla f_{n_k}(x_{n_k}) \limas_k \nabla f(\bar{x})$. Moreover, by \cref{as:unbiased.gradfn},
$\nabla f(x_{n_k}) \in \mathrm{m}\mathcal{F}_{n_k}$ and
\begin{align}
  \Expect_{\given\mathcal{F}_{n_k}} [\nabla f_{n_k}(x_{n_k})]
  & = \Expect_{\given\mathcal{F}_{n_k}} [\nabla f(x_{n_k})] + \varepsilon_{n_k}^f(x_{n_k}) 
    \notag \\
  & = \nabla f(x_{n_k}) + \varepsilon_{n_k}^f(x_{n_k})
    \limas_k \nabla f(\bar{x})\,. \label{expect.grad.fn.converges}
\end{align}

\begin{lemma}\label{lem:ranU.is.closed}
  The range space $\range U$ is closed in the strong topology of $\mathfrak{X}$, \ie,
  $\range U = \overline{\range}\,U$, where $\overline{\range}\,U$ denotes the smallest closed set
  containing $\range U$ (notice that
  $\mathfrak{X} \coloneqq L^2[(\Omega, \Sigma, \Prob), \mathcal{X}]$ is infinite dimensional).
\end{lemma}

\begin{proof}
  Since $\range U$ is finite dimensional within $\mathcal{X}$, there exists an orthonormal set
  $\Set{u_i}_{i=1}^{\rank U}$ which spans $\range U$. Hence, for any
  $z\in \mathfrak{X} \cap \range U$, there exist real-valued RVs $\Set{\gamma^i}_{i=1}^{\rank U}$
  s.t.\ $z = \sum_{i} \gamma^iu_i$ a.s. Due to the orthonormality of $u_i$s, it can be verified that
  $\norm{z}^2_{\mathfrak{X}} = \sum_i \Expect[(\gamma^i)^2]$. Thus,
  $z\in\mathfrak{X} \Rightarrow \gamma^i\in L^2[(\Omega, \Sigma, \Prob), \Real]$, $\forall
  i$. Consider, now, a sequence $(z_k)_k \subset \range U\cap\mathfrak{X}$, with the associated
  coefficients
  $\Set{\gamma_k^i \given i\in\Set{1, \ldots, \rank U}; k\in\IntegerP} \subset L^2[(\Omega, \Sigma,
  \Prob), \Real]$. Let $\bar{z}$ s.t.\ $z_k\xrightarrow{\mathfrak{X}}_k \bar{z}$. Since $(z_k)_k$ is
  convergent, it is also Cauchy~\cite{HB.PLC.book}, and thus, $(\gamma_k^i)_k$ is also Cauchy,
  $\forall i$. By virtue of the completeness of the Hilbert space
  $L^2[(\Omega, \Sigma, \Prob), \Real]$~\cite{HB.PLC.book}, there exists $\bar{\gamma}^i$ s.t.\
  $\gamma_k^i \xrightarrow{L^2[(\Omega, \Sigma, \Prob), \Real]}_k \bar{\gamma}^i$, $\forall i$. In
  other words,
  $\bar{z} = \lim_{k\to\infty} \sum_{i}\gamma_k^iu_i = \sum_{i} \lim_{k\to\infty}\gamma_k^iu_i =
  \sum_{i} \bar{\gamma}^iu_i \in \range U$, which establishes the claim.
\end{proof}

Since $(x_{n_k})_k$ is bounded a.s., \cref{as:as.bounded.x.as.bounded.xi} suggests that
$(\xi_{n_k})_k$ is also bounded a.s. There exists, thus, $\bar{\xi}$ and a subsequence of $(n_k)_k$,
denoted here also by $(n_k)_k$ to avoid notational congestion, s.t.\ $\xi_{n_k} \limas_k
\bar{\xi}$. Further, via
$\norm{w_{n_k+1}} \leq \lambda \norm{\xi_{n_k+1} + w_{n_k+1}/\lambda} + \lambda \norm{\xi_{n_k+1}}$
and \eqref{bounded.xi.w.as}, $(w_{n_k})_k$ is also bounded a.s., and hence, so is
$(\Expect_{\given \mathcal{F}_{n_k}} (w_{n_k}))_k$. Consequently, passing again to a subsequence of
$(n_k)_k$ if necessary, there exists $\bar{\mathsf{w}}$ s.t.\
$\Expect_{\given \mathcal{F}_{n_k}} (w_{n_k}) \limas_k \bar{\mathsf{w}}$.

Recall now that $(x_n)_n$ is bounded within $\mathfrak{X}$ [\cf discussion after
\eqref{Fejer.cond.type3}]. Moreover, the application of $\Expect(\cdot)$ to
\eqref{gradientf.converges}, \cref{as:bounded.var.gradfn}, and by arguments similar to those after
\eqref{gradientf.converges}, it can be shown that there exists $C_3\in\RealPP$ s.t.\
\begin{align}
  \norm{\nabla f_n(x_n)}_{\mathfrak{X}} \leq C_3\,, \quad \forall n\,. \label{bound.var.gradfn.xn}
\end{align}
Notice by \cref{as:bounded.var.pi} that $\forall n$,
$\norm{\pi_n-\pi_{n+1}}_{\mathfrak{X}}^2 = \norm{\pi_n- \pi + \pi - \pi_{n+1}}_{\mathfrak{X}}^2 \leq
2\norm{\pi_n-\pi}_{\mathfrak{X}}^2 + 2\norm{\pi_{n+1}-\pi}_{\mathfrak{X}}^2 \leq 4C_{\pi}$. Further,
$\norm{\pi_n}_{\mathfrak{X}}^2 = \norm{\pi_n - \pi + \pi}_{\mathfrak{X}}^2 \leq 2\norm{\pi_n -
  \pi}_{\mathfrak{X}}^2 + 2\norm{\pi}_{\mathfrak{X}}^2 \leq 2C_{\pi} +
2\norm{\pi}_{\mathfrak{X}}^2$; thus, $(\pi_n)_n$ is bounded. By \eqref{xi.n+1}, the a.s.\
nonexpansivity of $(Q_n)_n$ suggests that $\exists C_4\in \RealPP$ s.t.\
\begin{alignat*}{2}
  &&& \hspace{-2em} \norm{\xi_{n+1} + \tfrac{1}{\lambda} w_{n+1}}_{\mathfrak{X}} \notag\\
  & {} \leq {} && \tfrac{2\alpha-1}{\lambda} \norm{(\Id -
    T_{n+1})x_{n+1}}_{\mathfrak{X}} + \tfrac{1}{\lambda} \norm{x_n - x_{n+1}}_{\mathfrak{X}} \notag\\
  &&& + \tfrac{\alpha}{\lambda} \norm{(T_n - T_{n+1}) x_n}_{\mathfrak{X}} + \norm{\nabla
    f_n(x_n)}_{\mathfrak{X}} \\
  & \leq && \tfrac{2\alpha-1}{\lambda} \norm{x_{n+1}}_{\mathfrak{X}} + \tfrac{2\alpha-1}{\lambda}
  \norm{Q_{n+1}x_{n+1}}_{\mathfrak{X}} + \tfrac{2\alpha-1}{\lambda} \norm{\pi_{n+1}}_{\mathfrak{X}}\\
  &&& + \tfrac{1}{\lambda} \norm{x_n}_{\mathfrak{X}} + \tfrac{1}{\lambda}
  \norm{x_{n+1}}_{\mathfrak{X}} + \tfrac{\alpha}{\lambda} \norm{Q_nx_n}_{\mathfrak{X}} +
  \tfrac{\alpha}{\lambda} \norm{\pi_n}_{\mathfrak{X}} \\
  &&& + \tfrac{\alpha}{\lambda} \norm{Q_{n+1} x_n}_{\mathfrak{X}} + \tfrac{\alpha}{\lambda}
  \norm{\pi_{n+1}}_{\mathfrak{X}} + \norm{\nabla f_n(x_n)}_{\mathfrak{X}} \\
  & \leq && \tfrac{4\alpha-1}{\lambda} \norm{x_{n+1}}_{\mathfrak{X}} + \tfrac{2\alpha+1}{\lambda}
  \norm{x_n}_{\mathfrak{X}} + \tfrac{3\alpha-1}{\lambda} \norm{\pi_{n+1}}_{\mathfrak{X}} \\
  &&& + \tfrac{\alpha}{\lambda} \norm{\pi_n}_{\mathfrak{X}} + \norm{\nabla
    f_n(x_n)}_{\mathfrak{X}} \leq C_4\,.
\end{alignat*}
Due to \cref{as:l2.bounded.x.l2.bounded.xi}, which establishes the boundedness of $(\xi_{n_k})_k$,
the previous discussion renders $(w_{n_k})_k$ bounded. By Jensen's
inequality~\cite[\S9.7(h)]{Williams.book.91}, $(\Expect_{\given \mathcal{F}_{n_k}}(w_{n_k}))_k$ is
also bounded in $\mathfrak{X}$, and hence uniformly integrable
(UI)~\cite[\S13.3(a)]{Williams.book.91}. Since
$\Expect_{\given \mathcal{F}_{n_k}}(w_{n_k}) \limas_k\bar{\mathsf{w}}$, then, this convergence holds
also in probability~\cite[App.~A13.2(a)]{Williams.book.91}. This and the UI argument imply that
$\Expect_{\given \mathcal{F}_{n_k}}(w_{n_k}) \xrightarrow{\mathfrak{X}}_k
\bar{\mathsf{w}}$~\cite[App.~A13.2(f)]{Williams.book.91}. Going back to \eqref{w.expression}, notice
by \cref{lem:expect.inner.prod} and
$\Expect(\cdot) = \Expect[ \Expect_{\given
  \mathcal{F}_{n}}(\cdot)]$~\cite[\S9.7(a)]{Williams.book.91} that
$\forall u\in \ker U\cap \mathcal{X}$,
\begin{alignat}{2}
  &&& \hspace{-30pt} \innerp{u}{\Expect_{\given\mathcal{F}_{n_k}}(w_{n_k})}_{\mathfrak{X}} \notag\\
  & {} = {} && (1-\alpha) \innerp{u}{t_{n_k}}_{\mathfrak{X}} +
  \innerp{u}{U\Expect_{\given\mathcal{F}_{n_k}}(v_{n_k})}_{\mathfrak{X}} \notag\\
  & = && (1-\alpha) \innerp{u}{t_{n_k}}_{\mathfrak{X}} +
  \innerp{Uu}{v_{n_k}}_{\mathfrak{X}} \notag\\
  & = && (1-\alpha) \innerp{u}{t_{n_k}}_{\mathfrak{X}} = (1-\alpha) \innerp{u}{\Expect(t_{n_k})}
  \,. \label{tk.1st.step}
\end{alignat}
It can be also seen via \eqref{w.expression} that
$(1-\alpha) t_{n_k} = \Expect_{\given\mathcal{F}_{n_k}} (w_{n_k}) -
U\Expect_{\given\mathcal{F}_{n_k}}(v_{n_k})$; hence,
$(1-\alpha) \Expect(t_{n_k}) = \Expect (w_{n_k}) - U\Expect(v_{n_k})$ and
$(1-\alpha)^2 \norm{\Expect(t_{n_k})}^2 \leq 2\norm{\Expect (w_{n_k})}^2 + 2\norm{U}^2
\norm{\Expect(v_{n_k})}^2 \leq 2\norm{\Expect (w_{n_k})}^2 + 2
\norm{\Expect(v_{n_k})}^2$. Furthermore, Jensen's inequality~\cite[\S9.7(h)]{Williams.book.91}
yields
$(1-\alpha)^2 \norm{\Expect(t_{n_k})}^2 \leq 2\Expect(\norm{w_{n_k}}^2) + 2\Expect(\norm{v_{n_k}}^2)
= 2\norm{w_{n_k}}^2_{\mathfrak{X}} + 2\norm{v_{n_k}}^2_{\mathfrak{X}}$, and consequently, the
boundedness of $(w_{n_k})_k$ and $(v_{n_k})_k$ in $\mathfrak{X}$ results in that
$(\Expect(t_{n_k}))_k$ is also bounded in $\mathcal{X}$. According now to \cref{as:unbiased.Tn},
there exists a subsequence of $(n_k)_k$, denoted here again by $(n_k)_k$ to avoid clutter in
notations s.t.\ $\lim_k \Expect(t_{n_k})\in \range (\Id-Q) = \range U = (\ker U)^{\perp}$. Thus, via
\eqref{tk.1st.step} and the continuity of the inner product~\cite[Lem.~2.41(iii)]{HB.PLC.book},
$\innerp{u}{\bar{\mathsf{w}}}_{\mathfrak{X}} = \lim_k
\innerp{u}{\Expect_{\given\mathcal{F}_{n_k}}(w_{n_k})}_{\mathfrak{X}} = (1-\alpha) \innerp{u}{\lim_k
  \Expect(t_{n_k})} = 0$. Hence,
$\bar{\mathsf{w}}\in (\ker U)^{\perp} = \overline{\range}\,U = \range U$, according to
\cite[Fact~2.18(iii)]{HB.PLC.book} and \cref{lem:ranU.is.closed}.

Fix arbitrarily an $\epsilon > 0$. By the convexity of $h_{n_k}+g$, $\forall z\in\mathcal{X}$ and
a.s.,
\begin{alignat}{2}
  &&& \hspace{-2em} (h_{n_k}+g)(z) \notag\\
  & {} \geq {} && \innerp{z-x_{n_k+1}}{\xi_{n_k+1}} + (h_{n_k}+g)(x_{n_k+1}) \notag\\
  & \geq && \innerp{z-x_{n_k+1}}{\xi_{n_k+1}} + \innerp{x_{n_k+1} - x_{n_k}}{\tau_{n_k}} \notag\\
  &&& + (h_{n_k}+g)(x_{n_k}) \notag\\
  & = && \innerp{z-x_{n_k}}{\xi_{n_k+1}} + (h_{n_k}+g)(x_{n_k}) \notag\\
  &&& + \innerp{x_{n_k} - x_{n_k+1}}{\xi_{n_k+1}} + \innerp{x_{n_k+1} - x_{n_k}}{\tau_{n_k}} \notag\\
  & \geq && \innerp{z-x_{n_k}}{\xi_{n_k+1}} + (h_{n_k}+g)(x_{n_k}) \notag\\
  &&& - \norm{x_{n_k} - x_{n_k+1}} \norm{\xi_{n_k+1}} - \norm{x_{n_k} - x_{n_k+1}} \norm{\tau_{n_k}}
  \,, \label{show.esubgrad}
\end{alignat}
where $\tau_{n_k}\in \partial (h_{n_k}+g)(x_{n_k})$ is chosen according to \cref{as:bounded.tau},
$\forall k$. Moreover, by
\eqref{x.v.as.regular} and \cref{as:as.unbiasedness}, there exists an integer
$k_{\mypound} \coloneqq k_{\mypound}(\omega)$ s.t.\ $\forall k\geq k_{\mypound}$,
$\norm{x_{n_k} - x_{n_k+1}}\leq \epsilon/ [3(C_{\partial} + C_5)]$,
$-\epsilon/3 \leq \varepsilon_{n_k}^h(x_{n_k}) \leq \epsilon/3$ and
$-\epsilon/3 \leq -\varepsilon_{n_k}^h(z) \leq \epsilon/3$. By \eqref{show.esubgrad},
\begin{alignat*}{2}
  (h_{n_k}+g)(z) & {} \geq {} && \innerp{z-x_{n_k}}{\xi_{n_k+1}} +
  (h_{n_k}+g)(x_{n_k}) \\ 
  &&& - \tfrac{\epsilon/3}{C_{\partial} + C_5} (\norm{\tau_{n_k}} + \norm{\xi_{n_k+1}}) \,.
\end{alignat*}
Notice that \cref{as:as.bounded.x.as.bounded.xi} implies the existence of
$C_5 \coloneqq C_5(\omega)\in \RealPP$ s.t.\ $\norm{\xi_n}\leq C_5$. Applying
$\Expect_{\given \mathcal{F}_{n_k}}(\cdot)$ to the previous inequality and adhering to
\cref{as:unbiased.hn,as:bounded.tau}, as well as \cref{lem:expect.inner.prod},
$\forall z\in\mathcal{X}$, $\forall k\geq k_{\mypound}$ and a.s.,
\begin{alignat}{2}
  &&& \hspace{-20pt} \tfrac{\epsilon}{3} + (h+g)(z) \notag\\
  & {} \geq {} && -\varepsilon_{n_k}^h(z) + (h+g)(z) = \Expect_{\given \mathcal{F}_{n_k}}
  (h_{n_k}+g)(z) \notag\\
  & \geq && \innerp{z-x_{n_k}}{\Expect_{\given \mathcal{F}_{n_k}} (\xi_{n_k+1})} \notag\\
  &&& + \Expect_{\given \mathcal{F}_{n_k}} [(h_{n_k}+g)(x_{n_k})] - \tfrac{\epsilon}{3} \notag\\
  & = && \innerp{z-x_{n_k}}{\Expect_{\given \mathcal{F}_{n_k}} (\xi_{n_k+1})} + (h+g)(x_{n_k})
  \notag\\
  &&& - \varepsilon_{n_k}^h(x_{n_k}) - \tfrac{\epsilon}{3} \notag\\
  & \geq && \innerp{z-x_{n_k}}{\Expect_{\given \mathcal{F}_{n_k}} (\xi_{n_k+1})} + (h+g)(x_{n_k}) -
  \tfrac{\epsilon}{3} - \tfrac{\epsilon}{3} \,. \label{epsilon.subgrad.1st.step}
\end{alignat}
Since $(\xi_n)_n$ is bounded a.s., so is $(\xi_{n_k})_k$ and, consequently, so is
$(\Expect_{\given \mathcal{F}_{n_k}}(\xi_{n_k+1}))_k$. Hence, there exists $\bar{\xi}$ s.t.\
$\Expect_{\given \mathcal{F}_{n_k}} (\xi_{n_k+1}) \limas_k\bar{\xi}$ (once again, passing to a
subsequence of $(n_k)_k$ is avoided). Moreover, since $h+g$ is l.s.c.,
$\liminf_k (h+g)(x_{n_k}) \geq (h+g)(\bar{x})$~\cite{HB.PLC.book}. Hence, by the application of
$\liminf_k$ to \eqref{epsilon.subgrad.1st.step} and the continuity of the inner
product~\cite[Lem.~2.41(iii)]{HB.PLC.book},
$\epsilon/3 + (h+g)(z) \geq \innerp{z-\bar{x}}{\bar{\xi}} + (h+g)(\bar{x}) -2\epsilon/3$,
$\forall z\in\mathcal{X}$, and $(\bar{x}, \bar{\xi}) \in \graph \partial_{\epsilon} (h+g)$
a.s. Since $\epsilon > 0$ was chosen arbitrarily,
\begin{align}
  (\bar{x}, \bar{\xi}) \in \cap_{\epsilon\in\RealPP} \graph \partial_{\epsilon} (h+g) = \graph
  \partial (h+g) \quad\text{a.s.} \label{cluster.point.in.graph.hplusg}
\end{align}
Similarly to the way that \eqref{expect.wn.as.regular} follows from \eqref{wn.as.regular}, it can be
verified that \eqref{I.minus.Tn.to.zero} yields
$\Expect_{\given \mathcal{F}_{n_k}} \left[ (T_{n_k+1}-\Id) x_{n_k+1}\right] \limas_k 0$,
\eqref{Tn.as.regular} leads to
$\Expect_{\given \mathcal{F}_{n_k}} \left[(T_{n_k+1} - T_{n_k})x_{n_k}\right] \limas_k 0$, and
\eqref{x.v.as.regular} gives, via fact
$\norm{Q_n^{(\alpha)}} = \norm{\alpha Q_n + (1-\alpha) \Id}\leq \alpha \norm{Q_n} + (1-\alpha) \leq
1$, $\forall n$, a.s.,
$\Expect_{\given \mathcal{F}_{n_k}} [Q_{n_k+1}^{(\alpha)} (x_{n_k+1} - x_{n_k})] \limas_k
0$. Recalling \eqref{expect.wn.as.regular} and \eqref{expect.grad.fn.converges}, the application of
$\lim_{k} \Expect_{\given \mathcal{F}_{n_k}}(\cdot)$ to \eqref{gradf.and.ksi} yields%
\begin{alignat}{3}
  &&&&& \hspace{-1.25em} -\lim_{k\to\infty}  \Expect_{\given \mathcal{F}_{n_k}}(w_{n_k+1}) - \lambda
  \lim_{k\to\infty} \Expect_{\given \mathcal{F}_{n_k}} \left[\nabla f_{n_k}(x_{n_k})\right]
  \notag\\
  &&&&& - \lambda \lim\nolimits_{k\to\infty} \Expect_{\given \mathcal{F}_{n_k}} (\xi_{n_k+1}) 
  \notag\\ 
  &&& {} = {} && (1-2\alpha) \lim_{k\to\infty} \Expect_{\given \mathcal{F}_{n_k}} \left[
    (T_{n_k+1}-\Id) x_{n_k+1}\right] \notag\\
  &&&&& + \lim_{k\to\infty} \Expect_{\given \mathcal{F}_{n_k}} [Q_{n_k+1}^{(\alpha)} (x_{n_k+1} -
  x_{n_k})] \notag\\
  &&&&& + \alpha \lim_{k\to\infty} \Expect_{\given \mathcal{F}_{n_k}} \left[(T_{n_k+1} -
    T_{n_k})x_{n_k}\right] \notag\\ 
  & \Leftrightarrow {} &&&& \hspace{-1.25em} \nabla f(\bar{x}) + \bar{\xi}
  = -\tfrac{1}{\lambda} \bar{\mathsf{w}} \in \range U \quad\text{a.s.} \label{reach.optimality}
\end{alignat}
Since \eqref{reach.optimality} holds true for any cluster point in $\mathfrak{C}[(y_n)_n]$,
\cref{fact:Ypsilon.star} and \cref{lem:cluster.points} suggest that \textit{any}\/
$\bar{x}\in \mathfrak{C}[(x_n)_n]$ belongs to $\mathcal{A}_*$, solving thus
\eqref{the.problem} a.s.

\section{Proof of \cref{thm:exact.T}}\label[appendix]{app:exact.T}

\cref{as:as.bounded.x.as.bounded.xi,as:l2.bounded.x.l2.bounded.xi} are used in \cref{app:main.thm}
to establish the boundedness of $(w_n)_n$ a.s.\ and in $\mathfrak{X}$. However, in the case where
$T$ is known exactly, the boundedness of $(w_n)_n$ follows from the boundedness of $(v_n)_n$, since
by \eqref{define.w} and \eqref{define.v}, $w_n = Uv_n$. Moreover, by \cref{lem:expect.inner.prod}
and the fact that $v_n\in \mathrm{m}\mathcal{F}_n$,
$\Expect_{\given\mathcal{F}_n}(w_n) = U\Expect_{\given\mathcal{F}_n}(v_n) = Uv_n$. Thus,
$\bar{\mathsf{w}} \xleftarrow{\text{a.s.}}_k \Expect_{\given\mathcal{F}_{n_k}}(w_{n_k}) = Uv_{n_k}
\limas_k U\bar{v}$, and \eqref{reach.optimality} becomes
$\nabla f(\bar{x}) + \bar{\xi} = -(1/\lambda) U\bar{v}$ a.s. Hence, according to
\cref{fact:Ypsilon.star}, the arbitrarily chosen cluster point
$\bar{y} = (\bar{x}, \bar{v})\in \Upsilon^{(\lambda)}_*$. This result together with the stochastic
Fej\'{e}r monotonicity of $(y_n)_n$ w.r.t.\ $\Upsilon^{(\lambda)}_*$ [\cf \eqref{Fejer.cond}]
suggest that $\mathfrak{C}[(y_n)_n]$ is a singleton~\cite[Prop.~2.3(iv)]{PLC.JCP.stochastic.QF.15},
that $\mathfrak{C}[(x_n)_n]$ is also a singleton by virtue of \cref{lem:cluster.points}, and that
$(x_n)_n$ converges a.s.\ to a solution of \eqref{the.problem}.

\section{Proof of \cref{cor:HRLS}}\label[appendix]{app:HRLS}

According to \cref{lem:Tn.LS}, \cref{as:affine.maps} holds true. Since $(f, f_n) \coloneqq (0, 0)$
and $(h, h_n) \coloneqq (0, 0)$ in \eqref{HLS},
\crefnosort{as:losses,as:as.unbiasedness,as:bounded.var.gradfn} hold trivially true. Moreover,
\cref{lem:bounded.subgrads} and $(h, h_n) \coloneqq (0, 0)$ suggest that \cref{as:bounded.subgrads}
holds true. In the context of \cref{main.thm}, $L_{\nabla f}$ and $\lambda$ can take any values in
$\RealPP$. The claim of \cref{cor:HRLS} follows now from \cref{main.thm}.

\section{Proof of \cref{cor:S-FM-HSDM(CRegLS)}}\label[appendix]{app:S-FM-HSDM(CRegLS)}

Since $(f, f_n) \coloneqq (0, 0)$ in \eqref{CRegLS} and $T_n \coloneqq P_{\mathcal{A}} \eqqcolon T$
in \cref{algo:CRegLS}, \crefnosort{as:losses,as:unbiased.gradfn,as:bounded.var.gradfn} hold
trivially true. Moreover, according to \cref{lem:bounded.subgrads}, \cref{as:bounded.tau} holds also
true. The claim of \cref{cor:S-FM-HSDM(CRegLS)} follows now from \cref{thm:exact.T}.

\section{Proof of \cref{lem:Tn.LS}}\label[appendix]{app:Tn.LS}

The proof that \cref{as:Tn} holds true follows exactly the steps of the proof of \cite[(70a) and
(70d)]{FM-HSDM.Optim.18}, in the case where $\delta \coloneqq 1$ and
$\varphi_{\delta}(\vect{x}) = \varphi_1(\vect{x}) \coloneqq [1/(2n)] \sum\nolimits_{\nu=1}^n
(\vect{a}_{\nu}^{\intercal} \vect{x} - b_{\nu})^2$, $\forall \vect{x}\in \mathcal{X}$, in
\cite[(73)]{FM-HSDM.Optim.18}. Furthermore, by virtue of \cref{as:ergodic.an.bn}, of
($\varpi_n \coloneqq \varpi$, $\forall n$), and of the continuity of the matrix-inversion operation,
mappings \eqref{Tn.LS} satisfy \cref{as:consistent.Tn}. According to \cref{fact:U.I-Q}, the normal
equations suggest that for any $T\in\mathfrak{T}_{\mathcal{A}}$,
$\Set{\vect{x} \given \vect{R} \vect{x} = \vect{r}} = \ker(\Id-Q) + \bm{\theta}_* \Rightarrow
\ker(\Id-Q) = \Set{\vect{x} - \bm{\theta}_* \given \vect{R} \vect{x} = \vect{r} = \vect{R}
  \bm{\theta}_*} = \Set{\vect{x}' \given \vect{R} \vect{x}' = \vect{0}} = \ker\vect{R} =
\Set{\vect{0}}$, due to the non-singularity of $\vect{R}$. However,
$\range(\Id-Q) = [\ker(\Id-Q)]^{\perp} = \Set{\vect{0}}^{\perp} = \mathcal{X}$. Hence, if
$(\Expect(\vect{t}_n))_n$ is bounded, then \textit{any}\/ of its cluster points belongs trivially to
$\range(\Id-Q) = \mathcal{X}$.

\section{Proof of \cref{lem:unbiasedness.LS}}\label[appendix]{app:unbiasedness.LS}

Notice first that due to the IID assumption, $\forall\nu \in \Set{1, \ldots, n}$,
$\Expect_{\given \sigma(\vect{R}_n)}(\vect{a}_{\nu} \vect{a}_{\nu}^{\intercal}) =
\vect{R}_n$~\cite[\S9.11]{Williams.book.91}. Thus, by applying
$\Expect_{\given\sigma(\vect{R}_n)}(\cdot)$ to
$\vect{R}_{\nu} = (n/\nu)\vect{R}_n - (1/\nu) \sum_{i={\nu+1}}^n \vect{a}_i \vect{a}_i^{\intercal}$,
which can be straightforwardly derived from \eqref{Rn.rn},
$\Expect_{\given \sigma(\vect{R}_n)} (\vect{R}_{\nu}) = \vect{R}_n$ can be established. Moreover,
due to the conditional-independence hypothesis
$\Expect_{\given \mathcal{F}_n}(\vect{R}_{\nu}) = \Expect_{\given \mathcal{F}_n} \Expect_{\given
  \mathcal{F}_n, \sigma(\vect{R}_n)}[(1/\nu) \sum_{i=1}^{\nu} \vect{a}_i \vect{a}_i^{\intercal}] =
(1/\nu) \sum_{i=1}^{\nu} \Expect_{\given \mathcal{F}_n} \Expect_{\given \mathcal{F}_n,
  \sigma(\vect{R}_n)} (\vect{a}_i \vect{a}_i^{\intercal}) = (1/\nu) \sum_{i=1}^{\nu} \Expect_{\given
  \mathcal{F}_n} \Expect_{\given \sigma(\vect{R}_n)} (\vect{a}_i \vect{a}_i^{\intercal}) =
\Expect_{\given \mathcal{F}_n} (\vect{R}_n) = \vect{R}$. Furthermore, by the linear-regression model
of \cref{Sec:system.id} and the assumptions on the noise process $(\eta_n)_n$,
$\Expect_{\given \sigma(\vect{R}_n)} (b_{\nu}\vect{a}_{\nu}) = \Expect_{\given \sigma(\vect{R}_n)}
(\vect{a}_{\nu}\vect{a}_{\nu}^{\intercal} \bm{\theta}_*) + \Expect_{\given \sigma(\vect{R}_n)}
(\eta_n \vect{a}_{\nu}) = \Expect_{\given \sigma(\vect{R}_n)}
(\vect{a}_{\nu}\vect{a}_{\nu}^{\intercal}) \bm{\theta}_* + \Expect (\eta_n) \Expect_{\given
  \sigma(\vect{R}_n)}(\vect{a}_{\nu}) = \vect{R}_n\bm{\theta}_*$. Furthermore, in a way similar to
that of the $\Expect_{\given \mathcal{F}_n}(\vect{R}_{\nu})$ case,
$\Expect_{\given \mathcal{F}_n}(\vect{r}_{\nu}) = \Expect_{\given \mathcal{F}_n}(\vect{R}_n
\bm{\theta}_*) = \vect{R}\bm{\theta}_* = \vect{r}$.

\begin{applist}

\item\label[appendix]{app:unbiased.gradient} Applying $\Expect_{\given\mathcal{F}_n}(\cdot)$ to
  $(T-T_n)\vect{x}_n = - (\mu/\varpi) \bm{\mathcal{E}}_n^R\vect{x}_n + (\mu/\varpi)
  \bm{\varepsilon}_n^r$ yields \eqref{unbiased.Tn}. In a similar way, \eqref{unbiased.Qn} can be
  established. Moreover,
  \begin{alignat*}{2}
    \vect{t}_n
    && {} = {} & \Expect_{\given \mathcal{F}_n} \left[ \sum\nolimits_{\nu=1}^n (T-T_{\nu})
        x_{\nu} \right] \\
    && {} = {} & -\tfrac{\mu}{\varpi} \sum\nolimits_{\nu=1}^n \left[ \vect{R} - \Expect_{\given
        \mathcal{F}_n}(\vect{R}_{\nu}) \right]\vect{x}_{\nu} \\  
    &&& + \tfrac{\mu}{\varpi} \sum\nolimits_{\nu=1}^n \left[ \vect{r} - \Expect_{\given
        \mathcal{F}_n}(\vect{r}_{\nu}) \right] = \vect{0} \,.
  \end{alignat*}
  Furthermore,
  $\vect{R}_{n-1} - \vect{R}_n = \vect{R}_n/(n-1) - \vect{a}_n \vect{a}_n^{\intercal}/(n-1) =
  \vect{R}/(n-1) - \vect{a}_n \vect{a}_n^{\intercal}/(n-1) - \bm{\mathcal{E}}_n^R/(n-1)$ and
  $\vect{r}_{n-1} - \vect{r}_n = \vect{r}/(n-1) - b_n\vect{a}_n/(n-1) -
  \bm{\varepsilon}_n^r/(n-1)$. Hence, due to
  $T_n - T_{n-1} = (\mu/\varpi) (\vect{R}_{n-1} - \vect{R}_n) + (\mu/\varpi)(\vect{r}_n -
  \vect{r}_{n-1})$,
  \begin{alignat*}{2}
    &&& \hspace{-20pt} \Expect_{\given\mathcal{F}_n} [(T_n - T_{n-1})\vect{x}_{n-1}] \\
    && {} = {} & \tfrac{\mu}{\varpi(n-1)}\vect{R}\vect{x}_{n-1} - \tfrac{\mu}{\varpi(n-1)}
    \Expect_{\given\mathcal{F}_n} (\vect{a}_n \vect{a}_n^{\intercal})\vect{x}_{n-1} \\
    &&& - \tfrac{\mu}{\varpi(n-1)} \Expect_{\given\mathcal{F}_n}
    (\bm{\mathcal{E}}_n^R) \vect{x}_{n-1} - \tfrac{\mu}{\varpi(n-1)} \vect{r} \\
    &&& + \tfrac{\mu}{\varpi(n-1)} \Expect_{\given\mathcal{F}_n} (b_n\vect{a}_n) +
    \tfrac{\mu}{\varpi(n-1)} \Expect_{\given\mathcal{F}_n}(\bm{\varepsilon}_n^r) = \vect{0}\,.
  \end{alignat*}

\item By the assumption that noise $\eta_n$ is independent of $\vect{a}_n$, and thus independent
  also of $\mathcal{F}_n$, as well as by
  $\Expect_{\given\mathcal{F}_n}(\bm{\mathcal{E}}_n^R) = \vect{0}$, it can be verified that
  $\Expect_{\given \mathcal{F}_n}(b_n^2) = \Expect(b_n^2)$ a.s. Now, by \eqref{CRegLS} and
  \cref{lem:expect.inner.prod}, $\forall n$, $\forall \vect{x}\in \mathcal{X}$ and a.s.,
  \begin{align*}
    \varepsilon_n^h(\vect{x})
    & = \Expect_{\given \mathcal{F}_n}[(h-h_n)(\vect{x})] \\
    & = \tfrac{1}{2} \Expect_{\given \mathcal{F}_n} (\vect{x}^{\intercal} \bm{\mathcal{E}}_n^R \vect{x})
      - \Expect_{\given \mathcal{F}_n} (\vect{x}^{\intercal} \bm{\varepsilon}^r_n) \\
    & = \tfrac{1}{2}  \vect{x}^{\intercal} \Expect_{\given \mathcal{F}_n}(\bm{\mathcal{E}}_n^R) \vect{x}
      - \vect{x}^{\intercal} \Expect_{\given \mathcal{F}_n} (\bm{\varepsilon}^r_n) = 0 \,.
  \end{align*}
  The claim $\varepsilon_n^h(\vect{x}_n) = 0$, a.s., can be similarly verified.
  
\item Notice that
  $\Expect_{\given \mathcal{F}_n}[(\nabla f - \nabla f_n)\vect{x}_n] = \Expect_{\given \mathcal{F}_n}
  (\bm{\mathcal{E}}^R_n \vect{x}_n) - \Expect_{\given \mathcal{F}_n}(\bm{\varepsilon}^r_n) = \Expect
  (\bm{\mathcal{E}}^R_n) \vect{x}_n - \Expect(\bm{\varepsilon}^r_n) = \vect{0}$. Moreover, by
  following similar steps as those in the last part of \cref{app:unbiased.gradient},
  $\Expect_{\given \mathcal{F}_n}[(\nabla f_n - \nabla f_{n-1}) \vect{x}_{n-1}] = \vect{0}$ can be
  also established.

\item Here, only the case of $(h, h_n) = (l, l_n)$ is considered, since the case of
  $(h, h_n) = (0, 0)$ can be trivially deduced. To this end, there exists
  $\bm{\chi}_n\in \partial g(\vect{x}_n)$ s.t.\
  $\bm{\xi}_n = \vect{R}_{n-1} \vect{x}_n - \vect{r}_{n-1} + \bm{\chi}_n$. Recall here that
  $\bm{\chi}_n\in \partial g(\vect{x}_n) = \partial\norm{}_1(\vect{x}_n)$ iff
  \begin{align}
    [\bm{\chi}_n]_d \in
    \begin{cases}
      \Set{\sign([\vect{x}_n]_d)}\,, & \text{if}\ [\vect{x}_n]_d\neq 0\,, \\
      [-1,1]\,, & \text{if}\ [\vect{x}_n]_d = 0\,,
    \end{cases} \label{l1.subgrad}
  \end{align}
  where $\sign(\cdot)$ denotes the sign of a real-valued number. Hence,
  $\bm{\chi}_n\in \mathrm{m}\mathcal{F}_n$. By arguments similar to those in the first part of
  \cref{app:unbiased.gradient},
  $\Expect_{\given \mathcal{F}_n}(\bm{\xi}_n) = \Expect_{\given \mathcal{F}_n}(\vect{R}_{n-1})
  \vect{x}_n - \Expect_{\given \mathcal{F}_n}(\vect{r}_{n-1}) + \bm{\chi}_n = \vect{R}\vect{x}_n -
  \vect{r} + \bm{\chi}_n\in \partial (h+g)(\vect{x}_n)$. In other words, $\epsilon_n=0$, a.s., in
  \eqref{unbiased.subgrad.hg}.

\end{applist}

\section{Proof of \cref{lem:bounded.subgrads}}\label[appendix]{app:bounded.subgrads}

\begin{applist}

\item\label[appendix]{app:tau} According to \eqref{l1.subgrad}, for any $(\vect{z}, \bm{\tau})$ s.t.\
  $\bm{\tau} \in \partial\norm{}_1(\vect{z})$, $\lvert [\bm{\tau}]_d \rvert\leq 1$, $\forall
  d$. Hence, $\norm{\bm{\tau}}\leq \sqrt{D}$ and
  $\Expect_{\given \mathcal{F}_n} (\norm{\bm{\tau}}) \leq \sqrt{D}$ a.s. This bound renders
  \cref{as:bounded.tau,as:as.bounded.x.as.bounded.xi,as:l2.bounded.x.l2.bounded.xi} true in the case
  where $h_n \coloneqq 0$ and $h \coloneqq0$ in \eqref{CRegLS}.

  The following discussion deals with the case of $(h, h_n) \coloneqq (l, l_n)$ in \eqref{CRegLS},
  where, according to \eqref{grad.fn.LS}, $\nabla h_n(\vect{x}) = \vect{R}_n \vect{x} - \vect{r}_n$,
  $\forall \vect{x}\in \mathcal{X}$. By \cref{as:ergodic.an.bn}, given
  $\epsilon \coloneqq \epsilon(\omega) \in \RealPP$, there exists
  $n_{\mypound} \coloneqq n_{\mypound}(\omega) \in \IntegerP$ s.t.\
  $\norm{\vect{R}_n} \leq \norm{\vect{R}} + \epsilon$ and
  $\norm{\vect{r}_n} \leq \norm{\vect{r}} + \epsilon$, $\forall n\geq n_{\mypound}$. Define, then,
  $\varpi \coloneqq \varpi (\omega) \coloneqq \max \Set{ \Set{\norm{\vect{R}_n} \given 0\leq n<
      n_{\mypound}-1}, \norm{\vect{R}} + \epsilon}$ and
  $\varpi' \coloneqq \varpi'(\omega) \coloneqq \max \Set{ \Set{\norm{\vect{r}_n} \given 0\leq n<
      n_{\mypound}-1}, \norm{\vect{r}} + \epsilon}$. According to the hypothesis, there exists
  $C_z \coloneqq C_z(\omega) \in \RealPP$ s.t.\ $\norm{\vect{z}_n} \leq C_z$, $\forall n$. Thus,
  $\forall \bm{\delta}_n \in \partial\norm{}_1(\vect{z}_n)$,
  $\bm{\tau}_n \coloneqq \vect{R}_n \vect{z}_n - \vect{r}_n + \rho \bm{\delta}_n \in \partial
  (h_n+g)(\vect{z}_n)$ and
  $\norm{\bm{\tau}_n} \leq \norm{\vect{R}_n} \norm{\vect{z}_n} + \norm{\vect{r}_n} + \rho
  \norm{\bm{\delta}_n} \leq \varpi C_z + \varpi' + \rho \sqrt{D} \eqqcolon C_{\partial}$, $\forall n$;
  thus, \cref{as:bounded.tau} holds true. By using $\vect{x}_n$ in the place of $\vect{z}_n$ in the
  previous discussion, it can be verified that \cref{as:as.bounded.x.as.bounded.xi} holds also true.
  
\item Observe that
  $\norm{\bm{\xi}_n}^2 = \norm{\vect{R}_n \vect{x}_n - \vect{r}_n + \rho \bm{\delta}_n}^2 \leq
  2\norm{\vect{R}_n \vect{x}_n}^2 + 2\norm{\vect{r}_n + \rho \bm{\delta}_n}^2 \leq 2
  \norm{\vect{R}_n}^2 \norm{\vect{x}_n}^2 + 4\norm{\vect{r}_n}^2 + 4\rho^2 \norm{\bm{\delta}_n}^2
  \leq 2 \varpi^2 \norm{\vect{x}_n}^2 + 4\varpi'{}^2 + 4\rho^2 D$, a.s. Hence,
  $\norm{\bm{\xi}_n}^2_{\mathfrak{X}} = \Expect(\norm{\bm{\xi}_n}^2) \leq 2 \varpi^2 \Expect
  (\norm{\vect{x}_n}^2) + 4 \varpi'{}^2 + 4\rho^2 D = 2 \varpi^2 \norm{\vect{x}_n}_{\mathfrak{X}}^2
  + 4\varpi'{}^2 + 4\rho^2 D$, and consequently, $(\norm{\bm{\xi}_n}_{\mathfrak{X}}^2)_n$ is
  bounded.

\end{applist}




\end{document}